\documentclass{amsart}
\usepackage{amssymb,latexsym,graphics,enumerate,color}
\usepackage{fullpage}
\numberwithin{equation}{section}
\usepackage{bbm}
\usepackage[normalem]{ulem}

\usepackage{xcolor,cancel}

\newtheorem{theorem}{Theorem}[section]
\newtheorem{corollary}[theorem]{Corollary}
\newtheorem{lemma}[theorem]{Lemma}

\newcommand{\la}{\langle}
\newcommand{\ra}{\rangle}

\renewcommand{\O}{{\mathcal{O}}}

\newcommand{\R}{{\mathbb{R}}}

\newcommand{\cd}{{\,\cdot\,}}

\newcommand{\ang}{{\not\negmedspace\nabla}}
\newcommand{\good}{{\not\negmedspace\partial}}

\newcommand{\tC}{{\tilde{C}}}
\newcommand{\ttC}{{\tilde{\tilde{C}}}}

\newcommand{\tJ}{{\tilde{J}}}
\newcommand{\tK}{{\tilde{K}}}

\renewcommand{\S}{{\mathbb{S}}}

\begin{document}
\bibliographystyle{plain}

\title[Existence of quasilinear wave equations]
{
Long-time existence for systems of quasilinear wave equations
}

\author{Jason Metcalfe}
\address{Department of Mathematics, University of North Carolina,
  Chapel Hill, NC 27599-3250}
\email{metcalfe@email.unc.edu}
\author{Taylor Rhoads}
\address{Department of Mathematics, Vanderbilt University, Nashville,
  TN 37240}
\email{taylor.m.rhoads@vanderbilt.edu}

\thanks{The first author was supported in part by Simons Foundation
  Collaboration Grant 711724 and NSF grant DMS-2054910.}
 
\begin{abstract}
We consider quasilinear wave equations in $(1+3)$-dimensions where the nonlinearity
$F(u,u',u'')$ is
permitted to depend on the solution rather than just its derivatives.
For scalar equations, if $(\partial_u^2 F)(0,0,0)=0$, almost global
existence was established by Lindblad.  We seek to show a related
almost global existence result for coupled systems of such equations.
To do so, we will rely upon a variant of the $r^p$-weighted local energy
estimate of Dafermos and Rodnianski that includes a ghost weight akin
to those used by Alinhac.  The decay that is needed to close the
argument comes from space-time Klainerman-Sobolev type estimates from
the work of Metcalfe, Tataru, and Tohaneanu.
\end{abstract}

\maketitle

\section{Introduction}
In this article, we shall examine long-time existence for
systems of $(1+3)$-dimensional quasilinear wave equations with small
initial data where the nonlinearity is
permitted to depend on the solution rather than just its derivatives.
In particular, for $\Box = \partial_t^2-\Delta$, we shall examine
\begin{equation}
  \label{main_equation}
\begin{cases}
  \Box u^I = F^I(u,u',u''),\quad (t,x)\in \R_+\times \R^3,\quad
  I=1,2,\dots, M,\\
  u(0,\cd)=f,\quad \partial_tu(0,\cd)=g.
\end{cases}
\end{equation}
Here $u=(u^1,\dots, u^M)$.
We use the notation $u'=\partial u = (\partial_t u,\nabla u)$ for the
space-time gradient.
The smooth function $F$ vanishes to second order at the
origin, and it is linear in the $u''$ components.  Moreover, we shall
assume that
\begin{equation}
  \label{nou2}
  (\partial_{u^J} \partial_{u^K} F^I)(0,0,0)=0,\quad \forall\, I, J, K=1,2,\dots,M,
\end{equation}
which has the effect of disallowing $u^2$-type terms at the quadratic
level for $F$.  In \cite{Lindblad}, almost global existence,
which shows that the lifespan grows exponentially as the size of the
data shrinks, was proved for scalar equations.  In the current
article, we seek similar lower bounds on the lifespan for systems.

In the sequel, we use Einstein's summation convention.  Repeated Greek
letters $\alpha, \beta, \gamma$ are understood to sum from $0$ to $3$
where $x_0=t$.  Repeated lower case Roman letters $i, j, k$ are summed
from $1$ to $3$, and repeated upper case letters $I, J, K$ will be
summed from $1$ to $M$.

For simplicity, we shall truncate the nonlinearity at the quadratic
level:
\begin{equation}\label{quadratic}F^I(u,u', u'') = a^{I,\alpha}_{JK} u^J\partial_\alpha u^K +
  b^{I,\alpha\beta}_{JK}\partial_\alpha u^J \partial_\beta u^K +
  A^{I,\alpha\beta}_{JK} u^K \partial_\alpha\partial_\beta u^J +
  B^{I,\alpha\beta\gamma}_{JK} \partial_\gamma u^K
  \partial_\alpha\partial_\beta u^J.\end{equation}
In the small data regime, higher order terms are better behaved.
The constants will be assumed to satisfy the symmetry conditions
\begin{equation}\label{symmetry}A^{I,\alpha\beta}_{JK} = A^{I,\beta\alpha}_{JK} =
  A^{J,\alpha\beta}_{IK},\quad B^{I,\alpha\beta\gamma}_{JK} =
  B^{I,\beta\alpha\gamma}_{JK} = B^{J,\alpha\beta\gamma}_{IK}.
\end{equation}

Our main result is the following statement of almost global existence.
\begin{theorem}
  \label{main_thm}
Suppose that $f, g\in
(C^\infty_c(\R^3))^M$.  Moreover, assume that the smooth function $F$ vanishes to second order
at the origin, satisfies \eqref{nou2}, and is subject to the
symmetry conditions \eqref{symmetry}.  Then, for $N\in \mathbb{N}$ sufficiently
large, there are constants $c, \varepsilon_0>0$ so that if
  \begin{equation}
    \label{smallness}
     \sum_{|\mu|\le N+1} \|\partial^\mu f\|_{L^2} + \sum_{|\mu|\le N}
     \|\partial^\mu g\|_{L^2}\le \varepsilon
  \end{equation}
with $\varepsilon <\varepsilon_0$,
then \eqref{main_equation} has a unique solution with $u\in
(C^\infty([0,T_\varepsilon]\times \R^3))^M$ where 
\begin{equation}
  \label{lifespan}
  T_\varepsilon = \exp(c/\varepsilon^{\frac{1}{3}}).
\end{equation}
\end{theorem}

To aid the exposition, we have restricted to the case of compactly
supported initial data.  Without loss of generality, we shall take
\begin{equation}
  \label{WLOG}
  \text{supp}\, f^I, \text{supp}\, g^I\subset \{|x|\le 2\},\quad \forall \,
  I=1,2,\dots, M.
\end{equation}
  Small data in a sufficiently weighted space
would also suffice.  In \cite{Lindblad}, the lower bound of the lifespan was
$\exp(c/\varepsilon)$.  The difference in the lifespan between
\cite{Lindblad} and Theorem \ref{main_thm} is primarily
due to a logarithmic loss that occurs as a part of an endpoint Hardy
inequality.  See Lemma~\ref{lemma.Hardy}.   While refinements of our argument to improve
the power in \eqref{lifespan} are likely 
possible, it is not clear
what the sharp power is.

Equations such as \eqref{main_equation} with nonlinearity that depends
on the solution rather than just its derivatives do not mesh as simply
with the energy methods that are typically employed to prove long-time
existence.  In \cite{Hormander}, without the hypothesis \eqref{nou2},
a lower bound of $\exp(c/\varepsilon)$ was established in
$(1+4)$-dimensions.  For scalar equations, the additional hypothesis
\eqref{nou2} was, moreover, found to be sufficient to guarantee global
existence for sufficiently small data.  The analogous results in
$(1+3)$-dimensions appeared in \cite{Lindblad} where the lifespan was
shown to exceed $c/\varepsilon^2$ without \eqref{nou2} and almost
global existence was provided for scalar equations.

In both \cite{Hormander} and \cite{Lindblad}, the restriction to
scalar equations is necessitated by the use of the chain rule to write
$u\cdot \partial u = \frac{1}{2}\partial u^2$ interactions in
divergence form.  This special form, in turn, allows for easier
estimation of the solution rather than only its derivatives.  See,
e.g., \cite[Proposition 1.8]{Lindblad}.
The article
\cite{m_morgan} extended the result of \cite{Hormander} by establishing small data global existence for systems
\eqref{main_equation} subject to \eqref{nou2} in $(1+4)$-dimensions.
We establish the $(1+3)$-dimensional analog here.

The principle source of decay in our proof is obtained from space-time
Klainerman-Sobolev estimates as were proved in \cite{MTT}.  This will
be paired with variants of the integrated local energy decay
estimates.  In \cite{dafermos_rodnianski}, $r^p$-weighted local energy
estimates, which provide improved bounds on the ``good'' components of
the gradient: $\partial_t+\partial_r$ and $\ang := \nabla -
\frac{x}{r}\partial_r$, were proved.  These will be combined with a ghost weight,
which originates from \cite{Alinhac_ghostweight}.  This permits a
further improvement of the bounds in the vicinity of the light cone and
meshes particularly well with the space-time Klainerman-Sobolev
estimates of \cite{MTT}.

While the proof uses the method of invariant vector fields, it will
not rely on the Lorentz boosts.  While Lorentz boosts are perfectly
acceptable for systems such as \eqref{main_equation}, they can limit
further extensions to, e.g., multiple speed settings, equations in
exterior domains, or equations on stationary background geometries. 
Our proof is readily adaptable to Dirichlet wave equations exterior
to, say, star-shaped obstacles.  We do not include these extraneous
details here.  When combined with \cite{DZ}, \cite{DMSZ}, and
\cite{MS4}, it (largely) completes the extension of the long-time
existence results of \cite{Hormander} and \cite{Lindblad} to exterior
domains.  The work on systems in \cite{m_morgan} was also completed
exterior to star-shaped obstacles.  The restriction to star-shaped
obstacles is likely a convenience.  It is anticipated that any geometry
that permits a sufficiently rapid decay of local energy would
suffice.  See, e.g., \cite{HM4d, HM3d}.

The key aspects to the proof are to effectively bound the solution $u$
when typical energy methods estimate $\partial u$ and to obtain
additional decay from the derivative that must be present in at least
one factor of every nonlinear term.  In particular, it was the former
that restricted the analysis to scalar equations in preceding results.
The $r^p$-weights and ghost weights help with both aspects.  In
particular, when combined with rather standard Hardy-type estimates,
improved bounds on the local energy of the solution without
derivatives, when compared to e.g.~\eqref{le_pert}, result.  When
attempting to gain additional decay from the derivative that must
appear in at least one factor of every term of the nonlinearity, one
often relies on the scaling vector field, which near the light cone
gives additional decay in $t-|x|$.  The ghost weight allows us to take
advantage of this additional decay off of the light cone.

Since derivatives can be exchanged for extra decay using the scaling
vector field and since the $r^p$-weighted and ghost weighted estimates
allow for much larger weights for the good derivatives, which in
essence provides additional decay that can be used on other factors,
one can quickly become convinced that the worst possible nonlinear
terms are of the form $u (\partial_t-\partial_r)u$.  Moreover when all
of the vector fields land on the differentiated factor, one cannot
afford to lose the additional vector field that would result from
using the scaling vector field to get additional decay.  In this case,
we move the $\partial_t-\partial_r$ using integration by parts.
Within the local energy estimates, it could land on the weights or the
lower order factor, which allows us to gain additional decay, in the
latter case by using the scaling vector field.  Additionally it could
land on the multiplier, which has the basic form
$\partial_t+\partial_r$ and modulo better behaved terms results in
$\Box$ effectively turning these quadratic interactions into better cubic
interactions.

This article is organized as follows.  In the next subsection, we shall
gather some notation and preliminary results that will be used
frequently throughout the paper.  In Section 2, the integrated local
energy estimates will be proved.  Section 3 contains our sources of
decay, which are primarily space-time versions of the
Klainerman-Sobolev inequality.  Finally, the main theorem is proved in
Section 4.

\subsection{Notation}\label{notation}
 The vector fields that we rely on are
\[Z=(\partial_t,\partial_1,\partial_2,\partial_3,\Omega_1,\Omega_2,\Omega_3,S)\]
where
\[S=t\partial_t+r\partial_r,\quad   \Omega=x\times \nabla\]
represent the scaling vector field and generators of (spatial)
rotations, respectively.
We will frequently rely on the orthogonal decomposition
\[\nabla = \frac{x}{r}\partial_r + \ang,\]
and we shall use $\good = (\partial_t+\partial_r,\ang)$ as an
abbreviation for the ``good'' derivatives.  We note that
\[\ang = -\frac{x}{r^2}\times \Omega,\]
and as such,
\begin{equation}\label{angBound}|\ang u|\le \frac{1}{r}|Z u|.\end{equation}
Moreover, the following commutator will be used frequently in the
proofs of local energy estimates, as it was in the seminal work
\cite{Morawetz}:
\begin{equation}
  \label{angCommutator}
  [\nabla, \partial_r] = [\ang,\partial_r] = \frac{1}{r}\ang.
\end{equation}

The admissible vector fields are well-known to satisfy
\begin{equation}
  \label{commutators}
  [\Box, \partial]=[\Box, \Omega_{j}]=0,\quad [\Box, S]=2\Box.
\end{equation}
Moreover, we have
\begin{equation}\label{commutators2}[Z,\partial]\in
  \text{span}(\partial),\quad |[Z,\good]u|\le \frac{1}{r}|Zu| + |\good
  u|.
\end{equation}
We shall abbreviate
\[|Z^{\le N} u| = \sum_{|\mu|\le N} |Z^\mu u|,\quad |\partial^{\le N}
  u| = \sum_{|\mu|\le N} |\partial^\mu u|.\]

We use $L^p L^q$ as an abbreviation for $L^p_tL^q_x([0,T]\times
\R^3) = L^p([0,T]; L^q(\R^3))$.  In several circumstances, it will be
convenient to do an inhomogeneous dyadic decomposition of $\R^3$, and
for this purpose we denote
\[A_R=\{R<|x| <2R\},\quad
  \tilde{A}_R=\Bigl\{\frac{7}{8}R<|x|<\frac{17}{8}R\Bigr\}\quad
  \text{if }R>1,\]
and $A_1=\{|x|<2\}$, $\tilde{A}_1=\{|x|<17/4\}$.
For the standard integrated local energy estimates, 
we shall employ the following notations from 
\cite{Tataru_price}, \cite{MTT}:
\[\|u\|_{LE} = \sup_{j\ge 0} 2^{-j/2} \|u\|_{L^2L^2([0,T]\times A_{2^j})},\quad \|u\|_{LE^1} = \|(\partial u, |x|^{-1}u)\|_{LE}.\]

In the proof of the local energy estimates, we will often desire a $C^1(\R)$,
bounded, nondecreasing function and for these purposes set
\[\sigma_U(z)=\frac{z}{U+|z|},\quad U>0.\]
In the sequel, we shall also need dyadic decompositions in $t-r$, so
we set
\[X_U = \{(t,x)\in [0,T]\times \R^3\,:\, U<t-r<2U\}\: \text{for $U>1$}, \quad
  X_1=\{(t,x)\in [0,T]\times \R^3\,:\, |t-r|<2\},\]
with a similar enlargement being denoted by $\tilde{X}_U$.

The estimates from \cite{MTT} rely on a mixed decomposition where the cone is
divided in $|x|$ away from the light cone and in $t-|x|$ near the
cone.  We set $C=\{r\le t+2\}$.
Due to the simplifying assumption that the data are
compactly supported, it suffices to consider only this region, though
we believe that these estimates can be extended to all of $[0,T]\times
\R^3$ in a straightforward fashion.

We first divide into dyadic intervals in time
$C_\tau = \{t\in [\tau,2\tau]\cap [0,T], r\le t+2\}$.  Next, we shall decompose
dyadically in $r$ or $t-r$ depending on the proximity to the light
cone.  For $R, U>1$, we set
\[C^R_\tau = C_\tau \cap \{R<r<2R\},\quad C_\tau^U = C_\tau\cap
  \{U<t-r<2U\}\]
and
\[C^{R=1}_\tau = C_\tau \cap \{r<2\},\quad C^{U=1}_\tau = C_\tau
  \cap \{|t-r|<2\}.\]
We note that
\[C_\tau = \Bigl(\bigcup_{1\le R\le \tau/4} C_\tau^R\Bigr) \cup \Bigl(\bigcup_{1\le U\le
    \tau/4} C^U_\tau\Bigr)\cup C^{\frac{\tau}{2}}_\tau,\]
where
\[C^{\frac{\tau}{2}}_\tau = C_\tau\cap \{t-r\ge \frac{\tau}{2}\}\cap
  \{r\ge \frac{\tau}{2}\}.\]
We use $\tilde{C}^R_\tau$, $\tilde{C}_\tau^U$ to
denote enlargements of these sets on both the $R/U$ and $\tau$
scales.  In the latter case, we enlarge from $[\tau,2\tau]\cap [0,T]$ to
$[(7/8)\tau, 2\tau]\cap [0,T]$.  Subsequently $\tilde{\tilde{C}}^R_\tau$ and
$\tilde{\tilde{C}}^U_\tau$ will indicate further enlargements.  The key observation is that
\[\la r\ra \approx R,\quad t-r\approx \tau, \quad\text{ on } C^R_\tau,
  \tilde{C}^R_\tau, \tilde{\tilde{C}}^R_\tau,\quad \text{with } 1\le R\le \tau/4\]
and
\[r\approx \tau,\quad \la t-r\ra \approx U,\quad \text{ on }
  C^U_\tau, \tilde{C}^U_\tau, \tilde{\tilde{C}}^U_\tau,\quad
  \text{with } 1\le U\le \tau/4.\]
In this sense, the $C^{\frac{\tau}{2}}_\tau$ region can be thought of as either a
$C^R_\tau$ or a $C^U_\tau$ region.
Here and throughout, $R, U$ are understood to run over dyadic values.
Here we have set $\la r\ra$ to be a smooth function so that $\la
r\ra\ge 3$ and $\la
r\ra \approx r$ for $r\gg 1$.  For simplicity, we could simply take
$\la r\ra = (3+r)$ and $\la T \ra = (3+T)$.

In the sequel, we shall need cutoffs to localize to certain regions.
We fix $\beta\in C^\infty(\R)$ so that $\beta\equiv 1$ for $z<1$ and
$\beta \equiv 0$ for $z>2$.  We then set
\[\beta_{<R}(z)=\beta(z/R), \quad \beta_{>R}(z)=1-\beta_{<R}(z).\]


\section{Local energy estimate}
In this section, we shall collect some integrated local energy
estimates, which will serve as the main linear estimate for our proof
of almost global existence.  The first of these is a now standard version of the
original estimates of \cite{Morawetz}.  We proceed to explore a
variant of this that only bounds the ``good'' derivatives but with
much better weights.  What results is a mixture of the ghost weight
method of \cite{Alinhac_ghostweight} (see also the related estimates
in \cite{lindblad_rodnianski}) and the $r^p$-weighted local
energy estimates of \cite{dafermos_rodnianski}.

In this examination of a quasilinear problem, it will be helpful to
have estimates on small, time-dependent perturbations of the flat
operator $\Box$.  To this end, we define
\[(\Box_h u)^I = (\partial_t^2-\Delta)u^I -
h^{I,\alpha\beta}_J(t,x)\partial_\alpha\partial_\beta u^J.\]
Here we shall assume that
\begin{equation}
  \label{hsym}
  h^{I,\alpha\beta}_J = h^{I,\beta\alpha}_J =
  h^{J,\alpha\beta}_I, \qquad h^{I,\alpha\beta}_J\in C^1([0,T]\times
  \R^3).
\end{equation}
We shall also abbreviate
\[|h|=\sum_{I,J=1}^M\sum_{\alpha, \beta=0}^3
  |h^{I,\alpha\beta}_J|,\quad |\partial h| = \sum_{I,J=1}^M
  \sum_{\alpha,\beta,\gamma=0}^3 |\partial_\gamma h^{I,\alpha\beta}_J|.\]

We begin by recalling the local energy estimate for perturbations of
$\Box$ that was proved in \cite{ms_mathz}.  See, also,
\cite{Sterbenz}, \cite{MS}, \cite{mt_para, mt_exterior}, and \cite{MST}.
\begin{theorem}
  \label{thm_le_pert}  Suppose that $h$ satisfies \eqref{hsym}, that $|h|$
  is sufficiently small,
that $u\in (C^2([0,T]\times\R^3))^M$, and that for all $t\in [0,T]$, $|\partial^{\le 1}
u(t,x)|\to 0$ as $|x|\to \infty$.  Then,
  \begin{multline}
    \label{le_pert}
    \|\partial u\|^2_{L^\infty L^2} + \|u\|^2_{LE^1} \lesssim \|\partial
    u(0,\cd)\|^2_{L^2} + \int_0^T\int |\Box_h u|\Bigl(|\partial u| +
    \frac{|u|}{r}\Bigr)\,dx\,dt
\\    +\int_0^T \int\Bigl( |\partial
h| + \frac{|h|}{\la r\ra}\Bigr)
|\partial u| \Bigl(|\partial
    u|+\frac{|u|}{r}\Bigr)\,dx\,dt.
   \end{multline}
\end{theorem}

The proof of the theorem follows by pairing $\Box_h u$ with the
multiplier $C\partial_t u + \sigma_{2^j}(r)\partial_r u +
\frac{\sigma_{2^j}(r)}{r} u$, integrating over a space-time slab, and
integrating by parts. 

We next consider the following variant of Theorem~\ref{thm_le_pert}.
It represents a combination of the ideas of \cite{dafermos_rodnianski}
and \cite{Alinhac_ghostweight}.  The former considered multipliers
with principal part of the form $r^p (\partial_t+\partial_r) u$.
Rather than considering associated flux terms to bound terms similar
to the third term in the left side below, we shall instead modify the
multiplier using a ghost weight, which originates in
\cite{Alinhac_ghostweight}.  This more readily allows us to perform
necessary manipulations to prevent a loss of regularity due to the
quasilinear nature of the problem.  It will also allow us to
subsequently integrate by parts to control a term using ideas akin to
normal forms.

\begin{theorem}  Fix $0<p<2$.  Suppose that $h$ satisfies \eqref{hsym},
 that $u\in (C^2([0,T]\times \R^3))^M$, and that for
  all $t\in [0,T]$, $|r^{\frac{p+2}{2}}\partial^{\le 1}u(t,x)|\to 0$ as $|x|\to
  \infty$.  Then for any $U>0$,
  \begin{multline}
    \label{lep}
    \|r^{\frac{p-2}{2}} \good(ru)\|^2_{L^\infty L^2}
    + \|r^{\frac{p-3}{2}} \good(ru)\|^2_{L^2L^2}
+ \|r^{\frac{p-2}{2}} (\sigma'_U(t-r))^{\frac{1}{2}} (\partial_t+\partial_r)(ru) \|^2_{L^2L^2}
\\    \lesssim \|r^{\frac{p-2}{2}}\good(ru)(0,\cd)\|^2_{L^2}   
+ \sup_{t\in [0,T]} \int r^p |h| |\partial u|\Bigl(|\partial u| + \frac{|u|}{r}\Bigr)\,dx
\\+ \sup_{t\in [0,T]} \Bigl|\int_0^t \int r^{p-1} e^{-\sigma_U(t-r)} \Box_h
   u\cdot \Bigl(\partial_t + \partial_r\Bigr)(ru)\,dx\,dt\Bigr|
    +\int_0^T\int r^{p-1} \Bigl(|\partial h|+\frac{|h|}{r}\Bigr)
    |\partial u| |(\partial_t+\partial_r)(ru)|\,dx\,dt
\\+ \int_0^T\int r^{p-1} |h| |\partial u|\Bigl(|\ang
u|+\frac{|u|}{r}\Bigr)\,dx\,dt
+\int_0^T\int |h| r^{p-1} \sigma_U'(t-r) |\partial u|
|(\partial_t+\partial_r)(ru)|\,dx\,dt
\\
+\int_0^T\int r^p \Bigl(|(\partial_t+\partial_r)h|+\frac{|h|}{r}\Bigr)
|\partial u|^2\,dx\,dt.
  \end{multline}
The implicit constant is independent of both $T$ and $U$.
\end{theorem}

\begin{proof}
  We first note that
\begin{multline*}\int_0^t \int (\Box u)^I r^p e^{-\sigma_U(t-r)} \Bigl(\partial_t
  + \partial_r + \frac{1}{r}\Bigr)u^I\,dx\,dt
\\= \int_0^t \int_0^\infty \int_{\S^2} \Bigl(\partial_t^2-\partial_r^2 -
\ang\cdot\ang\Bigr)(ru)^I\cdot r^p e^{-\sigma_U(t-r)}
\Bigl(\partial_t+\partial_r\Bigr)(ru)^I\,d\omega\,dr\,dt.
\end{multline*}
Using integration by parts, we see that the right side is equal to
\begin{multline*}
\frac{1}{2}  \int_0^t\int_0^\infty \int_{\S^2} r^p
e^{-\sigma_U(t-r)}\Bigl(\partial_t-\partial_r\Bigr)\Bigl|\Bigl(\partial_t+\partial_r\Bigr)(ru)\Bigr|^2\,d\omega\,dr\,dt
\\+ \int_0^t\int_0^\infty\int_{\S^2} r^p e^{-\sigma_U(t-r)} \ang(ru)^I\cdot\ang\Bigl(\partial_t+\partial_r\Bigr)(ru)^I\,d\omega\,dr\,dt.
\end{multline*}
Relying on \eqref{angCommutator}, we
further see that this is the same as
\begin{multline*}
 \frac{1}{2} \int_0^\infty \int_{\S^2} r^pe^{-\sigma_U(t-r)}
  \Bigl|\Bigl(\partial_t+\partial_r\Bigr)(ru)\Bigr|^2\,d\omega\,dr\Bigl|_{t=0}^t
+\frac{p}{2}\int_0^t\int_0^\infty \int_{\S^2} r^{p-1}e^{-\sigma_U(t-r)}
\Bigl|\Bigl(\partial_t+\partial_r\Bigr)(ru)\Bigr|^2\,d\omega\,dr\,dt
\\+\int_0^t\int_0^\infty \int_{\S^2} r^p
\sigma_U'(t-r)e^{-\sigma_U(t-r)}
\Bigl|\Bigl(\partial_t+\partial_r\Bigr)(ru)\Bigr|^2\,d\omega\,dr\,dt
+ \int_0^t\int_0^\infty\int_{\S^2} r^{p-1} e^{-\sigma_U(t-r)}
|\ang(ru)|^2\,d\omega\,dr\,dt
\\+ \frac{1}{2}\int_0^t\int_0^\infty \int_{\S^2} r^p e^{-\sigma_U(t-r)}\Bigl(\partial_t+\partial_r\Bigr)|\ang(ru)|^2\,d\omega\,dr\,dt.
\end{multline*}
A final integration by parts then gives that
\begin{multline}\label{nopert}
 \int_0^t \int (\Box u)^I r^p e^{-\sigma_U(t-r)} \Bigl(\partial_t
  + \partial_r + \frac{1}{r}\Bigr)u^I\,dx\,dt\\=\frac{1}{2} \int_0^\infty \int_{\S^2} r^pe^{-\sigma_U(t-r)}\Bigl(
  \Bigl|\Bigl(\partial_t+\partial_r\Bigr)(ru)\Bigr|^2 + |\ang(ru)|^2\Bigr)\,d\omega\,dr\Bigl|_{t=0}^t
\\+\frac{p}{2}\int_0^t\int_0^\infty \int_{\S^2} r^{p-1}e^{-\sigma_U(t-r)}
\Bigl|\Bigl(\partial_t+\partial_r\Bigr)(ru)\Bigr|^2\,d\omega\,dr\,dt
\\+\Bigl(1-\frac{p}{2}\Bigr) \int_0^t\int_0^\infty\int_{\S^2} r^{p-1} e^{-\sigma_U(t-r)}
|\ang(ru)|^2\,d\omega\,dr\,dt
\\+\frac{1}{2}\int_0^t\int_0^\infty \int_{\S^2} r^p
\sigma_U'(t-r)e^{-\sigma_U(t-r)}
\Bigl|\Bigl(\partial_t+\partial_r\Bigr)(ru)\Bigr|^2\,d\omega\,dr\,dt.
\end{multline}
This string of equalities proves the desired estimate when $h\equiv
0$. 

We now consider the perturbation.  Using integration by parts, we note that
\begin{multline*}
  -\int_0^t\int h^{I,\alpha\beta}_J \partial_\alpha\partial_\beta u^J
  \cdot r^p
  e^{-\sigma_U(t-r)}\Bigl(\partial_t+\partial_r+\frac{1}{r}\Bigr)u^I\,dx\,dt
  = -\int h^{I,0\beta}_J r^{p-1} e^{-\sigma_U(t-r)}\partial_\beta u^J
  \Bigl(\partial_t+\partial_r\Bigr)(ru^I)\,dx\Bigl|_{t=0}^t
\\  +\int_0^t\int (\partial_\alpha h^{I,\alpha\beta}_J) r^{p-1}
  e^{-\sigma_U(t-r)}\partial_\beta
  u^J\Bigl(\partial_t+\partial_r\Bigr)(ru^I)\,dx\,dt
 \\ +\int_0^t\int h^{I,\alpha\beta}_J r^{-1}\partial_\alpha\Bigl(r^p
  e^{-\sigma_U(t-r)}\Bigr)\partial_\beta
  u^J\Bigl(\partial_t+\partial_r\Bigr)(ru^I)\,dx\,dt
  \\+\int_0^t\int h^{I,\alpha\beta}_J r^pe^{-\sigma_U(t-r)}\partial_\beta
  u^J \partial_\alpha\Bigl(\partial_t+\partial_r+\frac{1}{r}\Bigr)u^I\,dx\,dt.
\end{multline*}
Commuting the $\partial_\alpha$ using \eqref{angCommutator} and using the symmetries \eqref{hsym}, we see that
\begin{multline*}
  \int_0^t\int h^{I,\alpha\beta}_J r^pe^{-\sigma_U(t-r)}\partial_\beta
  u^J
  \partial_\alpha\Bigl(\partial_t+\partial_r+\frac{1}{r}\Bigr)u^I\,dx\,dt
  = \int_0^t \int h^{I,k\beta}_J r^{p-1}e^{-\sigma_U(t-r)}
  \partial_\beta u^J \ang_k u^I\,dx\,dt
  \\-\int_0^t\int h^{I,k\beta}_J r^{p-2} e^{-\sigma_U(t-r)}
  \frac{x_k}{r}u^I\partial_\beta u^J\,dx\,dt
  +\frac{1}{2} \int_0^t\int h^{I,\alpha\beta}_J r^p e^{-\sigma_U(t-r)}
  \Bigl(\partial_t+\partial_r+\frac{2}{r}\Bigr)\Bigl[\partial_\beta
  u^J \partial_\alpha u^I\Bigr]\,dx\,dt.
\end{multline*}
Combining the above two identities and integrating by parts, we obtain
\begin{multline}
  \label{withpert}
  -\int_0^t\int h^{I,\alpha\beta}_J \partial_\alpha\partial_\beta u^J
  \cdot r^p
  e^{-\sigma_U(t-r)}\Bigl(\partial_t+\partial_r+\frac{1}{r}\Bigr)u^I\,dx\,dt
 \\ = -\int h^{I,0\beta}_J r^{p-1} e^{-\sigma_U(t-r)}\partial_\beta u^J
  \Bigl(\partial_t+\partial_r\Bigr)(ru^I)\,dx\Bigl|_{t=0}^t
+\frac{1}{2}\int h^{I,\alpha\beta}_J r^p
e^{-\sigma_U(t-r)}\partial_\beta u^J\partial_\alpha u^I\,dx\Bigl|_{t=0}^t
  \\  +\int_0^t\int (\partial_\alpha h^{I,\alpha\beta}_J) r^{p-1}
  e^{-\sigma_U(t-r)}\partial_\beta
  u^J\Bigl(\partial_t+\partial_r\Bigr)(ru^I)\,dx\,dt
\\+p\int_0^t\int h^{I,k\beta}_J \frac{x_k}{r}r^{p-2}
e^{-\sigma_U(t-r)}\partial_\beta u^J
\Bigl(\partial_t+\partial_r\Bigr)(ru^I)\,dx\,dt
+\int_0^t \int h^{I,k\beta}_J r^{p-1}e^{-\sigma_U(t-r)}
  \partial_\beta u^J \ang_k u^I\,dx\,dt
  \\-\int_0^t\int h^{I,k\beta}_J r^{p-2} e^{-\sigma_U(t-r)}
  \frac{x_k}{r}u^I\partial_\beta u^J\,dx\,dt
+\int_0^t\int h^{I,\alpha\beta}_J \omega_\alpha r^{p-1} \sigma_U'(t-r)
e^{-\sigma_U(t-r)} \partial_\beta u^J (\partial_t+\partial_r)(ru^I)\,dx\,dt
  \\-\frac{1}{2}\int_0^t\int
  \Bigl(\partial_t+\partial_r\Bigr)(h^{I,\alpha\beta}_J) r^p
  e^{-\sigma_U(t-r)} \partial_\beta u^J \partial_\alpha u^I\,dx\,dt
  -\frac{p}{2}\int_0^t\int h^{I,\alpha\beta}_J
  r^{p-1}e^{-\sigma_U(t-r)}\partial_\beta u^J \partial_\alpha u^I\,dx\,dt.
\end{multline}
Here $\omega = (-1, x/r)$.
Our estimate \eqref{lep} is an immediate consequence of \eqref{nopert}
and \eqref{withpert} and taking the supremum over $t\in [0,T]$.
\end{proof}

We next consider a Hardy-type inequality to obtain associated bounds
on the solution, which are analogous to the bounds on
$\||x|^{-1}u\|_{LE}$ in \eqref{le_pert}.
\begin{lemma}
Let $0< p<2$.  Suppose $u\in C^1([0,T]\times \R^3)$ and that for
every $t\in [0,T]$, $|r^{p/2} u(t,x)|\to 0$ as $|x|\to \infty$.  Then
\begin{equation}
  \label{lot_R}
  \|r^{\frac{p-3}{2}} u\|_{L^2L^2} + \|r^{\frac{p-2}{2}} u\|_{L^\infty
      L^2}
\lesssim \|r^{\frac{p-2}{2}} u(0,\cd)\|_{L^2} + \|r^{\frac{p-3}{2}} \good(ru)\|_{L^2L^2}.
\end{equation}
\end{lemma}

\begin{proof}
By integrating by parts, for any $t\in [0,T]$, we have
\begin{align*}
\frac{1}{2-p}\int r^{p-2} u^2(t,x)\,dx &+  \int_0^t
  \int r^{p-3} u^2(\tau,x)\,dx\,d\tau\\
&= \frac{1}{2-p}\int r^{p-2} u^2(t,x)\,dx + \frac{1}{p-2} \int_0^t\int_{\S^2}
  \int_0^\infty (\partial_\tau+\partial_r)(r^{p-2})
  (ru)^2\,dr \,d\omega\,d\tau\\
&=\frac{1}{2-p}\int r^{p-2} u^2(0,x)\,dx -
  \frac{2}{p-2}\int_0^t\int_{\S^2}\int_0^\infty r^{p-2} (ru) (\partial_t+\partial_r)(ru)\,dr\,d\omega\,d\tau.
\end{align*}
As the Schwarz inequality allows us to bound the last term by
\[C \Bigl(\int_0^t\int r^{p-3}
  u^2\,dx\,d\tau\Bigr)^{1/2}\Bigl(\int_0^t\int
  r^{p-3} \Bigl[(\partial_t+\partial_r)(ru)\Bigr]^2\,dx\,d\tau\Bigr)^{1/2}\]
and as the first factor can be bootstrapped, \eqref{lot_R} follows
upon taking a supremum in $t\in [0,T]$.
\end{proof}

While we will not directly use the next lemma, which indicates the
form of the lower order bound with decay in $t-r$ and $0<p<1$,
we include it for the sake of completeness.
\begin{lemma}  Let $0\le p<1$.
Suppose $u\in C^1([0,T]\times\R^3)$ and that for every $t\in [0,T]$,
$|u(t,x)|\to 0$ as $|x|\to \infty$.   Then
  \begin{multline}
    \label{lot_U_p<1}
   \int_0^T\int r^{p-2} \sigma_U'(t-r) e^{-\sigma_U(t-r)} u^2\,dx\,dt +
   \sup_{t\in [0,T]} \int r^{p-1}\sigma_U'(t-r)e^{-\sigma_U(t-r)}
   u^2(t,x)\,dx
\\\lesssim \int r^{p-1}\sigma_U'(-r) e^{-\sigma_U(-r)} u^2(0,x)\,dx
+ \int_0^T\int r^{p-2} \sigma_U'(t-r)e^{-\sigma_U(t-r)} \Bigl[(\partial_t+\partial_r)(ru)\Bigr]^2\,dx\,dt.
  \end{multline}
\end{lemma}

\begin{proof}
 We argue similarly to the preceding lemma and apply integration by parts and
  the Schwarz inequality to observe that
  \begin{align*}
    \int_0^t&\int r^{p-2} \sigma_U'(\tau-r) e^{-\sigma_U(\tau-r)} u^2\,dx\,d\tau +
  \frac{1}{1-p} \int r^{p-1}\sigma_U'(t-r)e^{-\sigma_U(t-r)}u^2(t,x)\,dx\\
&= \frac{1}{p-1}\int_0^t\int_{\S^2}\int_0^\infty (\partial_\tau+\partial_r)[r^{p-1}
  \sigma_U'(\tau-r) e^{-\sigma_U(\tau-r)}] (ru)^2\,dr\,d\omega\,d\tau
    \\&\qquad\qquad\qquad\qquad\qquad\qquad\qquad\qquad\qquad\qquad\qquad\qquad+ 
  \frac{1}{1-p} \int r^{p-1}\sigma_U'(t-r)e^{-\sigma_U(t-r)}u^2(t,x)\,dx\\
&=\frac{1}{1-p}\int r^{p-1}\sigma_U'(-r)e^{-\sigma_U(-r)}u^2(0,x)\,dx
\\&\qquad\qquad\qquad\qquad\qquad\qquad-\frac{2}{p-1}\int_0^t\int_{\S^2}\int_0^\infty r^{p-1}\sigma_U'(\tau-r)e^{-\sigma_U(\tau-r)}
  ru(\partial_t+\partial_r)(ru)\,dr\,d\omega\,d\tau\\
&\lesssim \int r^{p-1}\sigma_U'(-r)e^{-\sigma_U(-r)}u^2(0,x)\,dx \\&\quad+ 
\Bigl(\int_0^t\int r^{p-2} \sigma_U'(\tau-r) e^{-\sigma_U(\tau-r)}
  u^2\,dx\,d\tau\Bigr)^{\frac{1}{2}}
\Bigl(\int_0^t\int r^{p-2} \sigma_U'(\tau-r)e^{-\sigma_U(\tau-r)} \Bigl[(\partial_t+\partial_r)(ru)\Bigr]^2\,dx\,d\tau \Bigr)^{\frac{1}{2}}.
  \end{align*}
We may then bootstrap the first factor of the last term and take a
supremum in $t$ to complete the argument.
\end{proof}

We shall need the analog of the above when $p=1$, which comes with a
logarithmic loss.  It is this logarithm that is largely responsible
for the difference between \eqref{lifespan} and the
$\exp(c/\varepsilon)$ lifespan of \cite{Lindblad}.
\begin{lemma} \label{lemma.Hardy} Suppose $u\in C^1([0,T]\times\R^3)$
  and that for every $t\in [0,T]$, 
$|u(t,x)|\to 0$ as $|x|\to \infty$.   Then
  \begin{multline}
    \label{lot_U_p=1}
    \int_0^T\int \beta_{>2}(r) \frac{1}{r (\log \la r\ra)^2} 
    \sigma_U'(t-r)e^{-\sigma_U(t-r)} u^2\,dx\,dt \\+ \sup_{t\in [0,T]} \int
   \beta_{>2}(r) \frac{1}{\log \la r\ra} \sigma_U'(t-r)e^{-\sigma_U(t-r)}
    u^2(t,x)\,dx
\\\lesssim \int \frac{1}{\log \la r\ra} \sigma_U'(-r)e^{-\sigma_U(-r)}
    u^2(0,x)\,dx
\\+ \int_0^T\int r^{-1} \sigma_U'(t-r) e^{-\sigma_U(t-r)}
\Bigl[\Bigl(\partial_t+\partial_r\Bigr)(ru)\Bigr]^2\,dx\,dt
+\|u\|^2_{LE^1}.
  \end{multline}
\end{lemma}

\begin{proof}
 We observe that 
\begin{align*} \int_0^t\int &\beta_{>2}(r) \frac{1}{r(\log r)^2}
    \sigma_U'(\tau-r)e^{-\sigma_U(\tau-r)} u^2\,dx\,d\tau + \int \beta_{>2}(r)
    \frac{1}{\log r} \sigma_U'(t-r)e^{-\sigma_U(t-r)}
    u^2(t,x)\,dx
\\&= \int_0^t\int_{\S^2}\int_0^\infty \beta_{>2}(r)\Bigl(\partial_\tau+\partial_r\Bigr)\Bigl[-\frac{1}{\log r}
    \sigma_U'(\tau-r)e^{-\sigma_U(\tau-r)}\Bigr] (r u)^2\,dr\,d\omega\,d\tau
  \\&\qquad\qquad\qquad\qquad\qquad\qquad\qquad\qquad  + \int \beta_{>2}(r)
    \frac{1}{\log r} \sigma_U'(t-r)e^{-\sigma_U(t-r)}
    u^2(t,x)\,dx
\\&=2 \int_0^t\int_{\S^2}\int_0^\infty \beta_{>2}(r) \frac{1}{\log r}
    \sigma_U'(\tau-r)e^{-\sigma_U(\tau-r)} (r u)
    \Bigl(\partial_t+\partial_r\Bigr)(ru)\,dr\,d\omega\,d\tau 
\\&\qquad\qquad+ \int
    \beta_{>2}(r) \frac{1}{\log r} \sigma_U'(-r)e^{-\sigma_U(-r)}
    u^2(0,x)\,dx
+\int_0^t\int \beta_{>2}'(r) \frac{1}{\log r}\sigma_U'(\tau-r)
    e^{-\sigma_U(\tau-r)} u^2\,dx\,d\tau.
\end{align*}
As $\text{supp}\,\beta'_{>2}\subset [2,4]$, the last term is bounded by $\|u\|_{LE^1}^2$.
The Schwarz inequality shows that the first term in the right side is
\begin{multline*}\lesssim \Bigl(\int_0^t\int \beta_{>2}(r) \frac{1}{r(\log r)^2}
  \sigma_U'(\tau-r) e^{-\sigma_U(\tau-r)} u^2\,dx\,d\tau\Bigr)^{1/2}\\\times
  \Bigl(\int_0^t\int \beta_{>2}(r) r^{-1} \sigma_U'(\tau-r) e^{-\sigma_U(\tau-r)}
  \Big[\Bigl(\partial_t+\partial_r\Bigr)(ru)\Bigr]^2\,dx\,d\tau\Bigr)^{1/2}.
\end{multline*}
The first factor may be bootstrapped, and upon taking a supremum over
$t\in [0,T]$, the proof is complete. 
\end{proof}

The following corollary combines \eqref{lep}, \eqref{lot_R}, and
\eqref{lot_U_p=1} when $p=1$ with \eqref{commutators} and provides the primary
linear estimate that our proof is based upon.
\begin{corollary}  Fix $N\in \mathbb{N}$.
 Suppose that $h$ satisfies \eqref{hsym},
 that $u\in C^2([0,T]\times \R^3)$, and that for
  all $t\in [0,T]$, $|r^{\frac{p+2}{2}}\partial^{\le 1}Z^{\le N} u(t,x)|\to 0$ as $|x|\to
  \infty$.  Then, 
  \begin{multline}
    \label{newLE}
   \|\la r\ra^{\frac{1}{2}} \good Z^{\le N} u\|^2_{L^\infty L^2} +
   \| r^{-\frac{1}{2}}
   Z^{\le N} u\|_{L^\infty L^2}^2
+ \|\good
  Z^{\le N} u\|^2_{L^2L^2} 
 + \|r^{-1} Z^{\le N} u\|_{L^2L^2}^2
\\+ \sup_{U\ge 1} \Bigl(\|r^{-\frac{1}{2}} \la
t-r\ra^{-\frac{1}{2}}(\partial_t+\partial_r)(rZ^{\le N} u)\|^2_{L^2L^2(X_U)}
+ \|r^{-\frac{1}{2}} (\log \la r\ra)^{-1} \la t-r\ra^{-\frac{1}{2}}
Z^{\le N}u\|^2_{L^2L^2(X_U)}\Bigr)
\\
 \lesssim \|r^{-\frac{1}{2}}(r\good)^{\le 1} Z^{\le N} u(0,\cd)\|^2_{L^2}   
+ \sup_{t\in [0,T]} \int r |h| |\partial Z^{\le N} u|\Bigl(|\partial
Z^{\le N} u| + \frac{|Z^{\le N} u|}{r}\Bigr)\,dx
\\+ \sup_{U\ge 1} \sup_{t\in [0,T]} \Bigl|\int_0^t \int e^{-\sigma_U(t-r)} \Box_h
  Z^{\le N} u\cdot \Bigl(\partial_t + \partial_r\Bigr)(r Z^{\le N} u)\,dx\,dt\Bigr|
  \\  +\int_0^T\int\Bigl(|\partial h|+\frac{|h|}{r}\Bigr)
    |\partial Z^{\le N} u| |(\partial_t+\partial_r)(r Z^{\le N}u)|\,dx\,dt
+ \int_0^T\int |h| |\partial Z^{\le N} u|\Bigl(|\ang Z^{\le N}
u|+\frac{|Z^{\le N}u|}{r}\Bigr)\,dx\,dt
\\+\int_0^T\int |h|\la t-r\ra^{-1} |\partial Z^{\le N} u|
|(\partial_t+\partial_r)(rZ^{\le N}u)|\,dx\,dt
+\int_0^T\int r \Bigl(|(\partial_t+\partial_r)h|+\frac{|h|}{r}\Bigr)
|\partial Z^{\le N}u|^2\,dx\,dt
\\+\|\partial Z^{\le N}u\|^2_{L^\infty L^2} + \|Z^{\le N}u\|^2_{LE^1}.
  \end{multline}
\end{corollary}

\begin{proof}
Using \eqref{lot_R} and \eqref{lot_U_p=1}, we may adapt \eqref{lep}
with $p=1$ to
the bound
\begin{multline}
    \label{lep2}
    \|\la r\ra^{\frac{1}{2}} \good Z^{\le N} u\|^2_{L^\infty L^2} + \|r^{-\frac{1}{2}}
    Z^{\le N} u\|^2_{L^\infty L^2}
    + \|\good Z^{\le N} u\|^2_{L^2L^2} + \| r^{-1} Z^{\le N} u\|^2_{L^2L^2}
\\+ \|r^{-\frac{1}{2}} (\sigma'_U(t-r))^{\frac{1}{2}}
(\partial_t+\partial_r)(rZ^{\le N} u) \|^2_{L^2L^2}
+ \|r^{-\frac{1}{2}}(\log(\la r\ra)^{-1}
(\sigma_U'(t-r))^{\frac{1}{2}} Z^{\le N} u\|^2_{L^2L^2}
\\    \lesssim \|r^{-\frac{1}{2}}(r\good)^{\le 1} Z^{\le N} u(0,\cd)\|^2_{L^2}   
+ \sup_{t\in [0,T]} \int r |h| |\partial Z^{\le N} u|\Bigl(|\partial
Z^{\le N} u| + \frac{|Z^{\le N} u|}{r}\Bigr)\,dx
\\+ \sup_{t\in [0,T]} \Bigl|\int_0^t \int e^{-\sigma_U(t-r)} \Box_h
  Z^{\le N} u\cdot \Bigl(\partial_t + \partial_r\Bigr)(r Z^{\le N} u)\,dx\,dt\Bigr|
  \\  +\int_0^T\int\Bigl(|\partial h|+\frac{|h|}{r}\Bigr)
    |\partial Z^{\le N} u| |(\partial_t+\partial_r)(r Z^{\le N}u)|\,dx\,dt
+ \int_0^T\int |h| |\partial Z^{\le N} u|\Bigl(|\ang Z^{\le N}
u|+\frac{|Z^{\le N}u|}{r}\Bigr)\,dx\,dt
\\+\int_0^T\int |h| \sigma_U'(t-r) |\partial Z^{\le N} u|
|(\partial_t+\partial_r)(rZ^{\le N}u)|\,dx\,dt
+\int_0^T\int r \Bigl(|(\partial_t+\partial_r)h|+\frac{|h|}{r}\Bigr)
|\partial Z^{\le N}u|^2\,dx\,dt
\\+\|\partial Z^{\le N}u\|^2_{L^\infty L^2} + \|Z^{\le N}u\|^2_{LE^1}.
  \end{multline}
We note that
\[\sigma'_U(t-r)\gtrsim \frac{1}{\la t-r\ra} \text{ on }
  X_U,\quad\text{ and }\quad
\sigma'_U(t-r)\lesssim \frac{1}{\la t-r\ra} \text{ provided }
  U\ge 1.\]
Using these facts in \eqref{lep2} and subsequently taking a supremum
in $U$ yields \eqref{newLE}.
\end{proof}

This last Hardy-type inequality is not strictly necessary.  When we
set up an iteration to solve \eqref{main_equation}, this will be a
convenience when showing that the sequence converges.  In particular,
it will allow us to focus only on energy and integrated local energy
spaces for this portion of the argument.  A closely related
calculation appears in \cite{Lindblad}.

\begin{lemma}
  Suppose that $u\in C^1([0,T]\times \R^3)$ is supported where $\{r\le
  t+2\}$.  Then
  \begin{equation}
    \label{lastHardy}
   \int \frac{1}{(1+r)(t-r+3)^2} u^2\,dx \lesssim \int
   \frac{1}{(1+r)r^2}u^2\,dx + \int \frac{1}{(1+r)}(\partial_r u)^2\,dx.
  \end{equation}
\end{lemma}

\begin{proof}
  For $t\in [0,T]$ fixed, we integrate by parts and apply the Schwarz
  inequality to obtain
  \begin{align*}
    \int \frac{1}{(1+r)(t-r+3)^2}& u^2\,dx=
                                            \int_{\S^2}\int_0^\infty \partial_r[(t-r+3)^{-1}]
                                            \frac{r^2}{(1+r)} u^2\,dr\,d\omega\\
&= - \int \frac{1}{t-r+3}\cdot\frac{2+r}{r(1+r)^2} u^2\,dx - 2\int \frac{1}{(1+r)(t-r+3)}
  u\,\partial_r u\,dx\\
&\lesssim \Bigl(\int \frac{1}{(1+r)(t-r+3)^2}
  u^2\,dx\Bigl)^{\frac{1}{2}}
\Bigl[\Bigl(\int \frac{1}{r^2(1+r)} u^2\,dx\Bigr)^{\frac{1}{2}} +
  \Bigl(\int \frac{1}{(1+r)} (\partial_r u)^2\,dx\Bigr)^{\frac{1}{2}}\Bigr].
  \end{align*}
Dividing both sides by the first factor in the right completes the proof.
\end{proof}


\section{Sobolev estimates}
The main decay estimate that we shall rely upon is a space-time variant of the
Klainerman-Sobolev estimate \cite{klainermanSob} that was established in \cite{MTT}
and is particularly well-adapted to integrated local energy estimates.

As is described in Section~\ref{notation}, we will break space-time up into $C^R_\tau$ and $C^U_\tau$ regions
where $\tau\in [0,T]$ and $1\le R,U\le \tau/4$.
On these regions, we have the following weighted Sobolev estimates,
which will serve as our source of decay.
\begin{lemma}\label{lemma_spacetime_KS}
  For any $\tau \ge 1$ and $1\le R, U\le \tau/4$, we have
  \begin{align}
    \label{crt}
 \|w\|_{L^\infty L^\infty (C^R_\tau)} &\lesssim
    \frac{1}{\tau^{\frac{1}{2}}R^{\frac{3}{2}}} 
    \|Z^{\le 4} w\|_{L^2L^2(\tilde{C}^R_\tau)} +
    \frac{1}{\tau^{\frac{1}{2}}R^{\frac{1}{2}}}
    \|(\partial_t+\partial_r) Z^{\le 3} w\|_{L^2L^2(\tilde{C}^R_\tau)},\\
\label{cut}  \|w\|_{L^\infty L^\infty(C_\tau^U)}&\lesssim
                                                  \frac{1}{\tau^{\frac{3}{2}}
                                                  U^{\frac{1}{2}}}\|Z^{\le
                                                  4}w\|_{L^2L^2(\tilde{C}^U_\tau)}
                                                  +
                                                  \frac{1}{\tau^{\frac{3}{2}}
                                                  U^{\frac{1}{2}}}
                                                  \|(\partial_t+\partial_r)
                                                  (rZ^{\le 3}
                                                  w)\|_{L^2L^2(\tilde{C}^U_\tau)},\\
    \label{ctt} \|w\|_{L^\infty L^\infty (C^{\frac{\tau}{2}}_\tau)} &\lesssim
    \frac{1}{\tau^{2}} 
    \|Z^{\le 4} w\|_{L^2L^2(\tilde{C}^{\frac{\tau}{2}}_\tau)} +
    \frac{1}{\tau}
    \|(\partial_t+\partial_r) Z^{\le 3} w\|_{L^2L^2(\tilde{C}^{\frac{\tau}{2}}_\tau)}
  \end{align}
\end{lemma}

See \cite[Lemma 3.8]{MTT}.  The proof of \eqref{crt} follows by
changing coordinates to $t=e^s$, $r=e^{s+\rho}$ and applying Sobolev
  embeddings in $\omega$ and the Fundamental Theorem of Calculus in
  $s$ and $\rho$.  In fact, this yields
  \begin{equation}
    \label{crt1}
    \|w\|_{L^\infty L^\infty(C^R_\tau)} \lesssim
    \frac{1}{\tau^{\frac{1}{2}}R^{\frac{3}{2}}} \|Z^{\le 3}
    w\|_{L^2L^2(\tilde{C}^R_\tau)} + \frac{1}{\tau^{\frac{1}{2}}R}
    \|Z^{\le 3} w\|^{1/2}_{L^2L^2(\tilde{C}^R_\tau)} \|\partial_r
    Z^{\le 3} w\|^{1/2}_{L^2L^2(\tilde{C}^R_\tau)}
  \end{equation}
  for any $R>1$.
In order to get additional decay out of differentiated terms, such as
those appearing in the last term of \eqref{crt1}, the preceding work \cite{m_morgan} in $(1+4)$-dimensions
relied upon \cite[Lemma 3.11]{MTT}.  As \eqref{newLE} provides better
control on the good derivatives, we can argue more simply and instead
use
\begin{equation}
  \label{derivative_decay}
  (\partial_t-\partial_r) = \frac{2}{t-r} S - \frac{t+r}{t-r}(\partial_t+\partial_r).
\end{equation}
As $(t+r)/(t-r) = \O(1)$ on $\tilde{C}^R_\tau$ with $R\le \tau/4$,
\eqref{crt} follows immediately.  Replacing $w$ by
$\beta_{>\tau/2}(t-r) w$ in \eqref{crt1} and using that
$S(\beta_{>\tau/2}(t-r))=\O(1)$, \eqref{ctt} is obtained similarly.

When $U=1$, the other estimate \eqref{cut} is an immediate corollary
of \eqref{crt1} as we need only consider $\partial_r$ as a vector field.
When $U>1$, 
\begin{equation}\label{cut2}\|w\|_{L^\infty L^\infty (C^U_\tau)} \lesssim
  \frac{1}{\tau^{\frac{3}{2}} U^{\frac{1}{2}}} \|Z^{\le 3}
  w\|_{L^2L^2(\tilde{C}^U_\tau)} +
  \frac{U^{\frac{1}{2}}}{\tau^{\frac{3}{2}}} \|\partial Z^{\le 3}
  w\|_{L^2L^2(\tilde{C}^U_\tau)}
\end{equation}
follows from arguments similar to the above in coordinates $t+r=e^s$,
$t-r=e^{s+\rho}$.  Subsequently applying \eqref{derivative_decay}
yields \eqref{cut}.

It will also be helpful to have the following common weighted Sobolev
estimate of \cite{Klainerman}.
\begin{lemma}  Provided that $h\in C^\infty(\R^3)$,
 \begin{equation}
    \label{weighted_Sobolev}
    \|h\|_{L^\infty(A_R)} \lesssim R^{-1} \|Z^{\le 2} h\|_{L^2(\tilde{A}_R)}.
  \end{equation}
\end{lemma}
For $R=1$, standard Sobolev embeddings yield the result.  And for
$R>1$, after localizing, one only needs to apply Sobolev embeddings in 
$(r,\omega)$.  The decay is then a consequence of converting the
volume element $dr\,d\omega$ to $dx = r^2\,dr\,d\omega$.


  \section{Proof of Theorem~\ref{main_thm}}
We shall solve \eqref{main_equation} via iteration.  We set $u_0\equiv
0$, and for $k\ge 1$, let $u_k$ solve
\begin{equation}
  \label{iteration_equation}
\begin{cases}
  \Box u_k^I =a^{I,\alpha}_{JK} u_{k-1}^J\partial_\alpha u_{k-1}^K +
  b^{I,\alpha\beta}_{JK}\partial_\alpha u_{k-1}^J \partial_\beta u_{k-1}^K +
  A^{I,\alpha\beta}_{JK} u_{k-1}^K \partial_\alpha\partial_\beta u_k^J +
  B^{I,\alpha\beta\gamma}_{JK} \partial_\gamma u_{k-1}^K
  \partial_\alpha\partial_\beta u_k^J,\\
  u_k(0,\cd)=f,\quad \partial_tu_k(0,\cd)=g.
\end{cases}
\end{equation}
Note that the right side of the equation is $F^I(u_{k-1}, u'_{k-1},
u''_k)$.  We will show that the sequence $(u_k)$ converges.  The limit
of this sequence is the desired solution.

We shall work with $N=60$, though this is far from sharp.  In the
$r^p$-weighted estimates, we use $p=1$ throughout.

To show that the sequence is convergent, we first show a
certain boundedness.  Relying on that, we next demonstrate that the sequence
is Cauchy, which due to completeness of the spaces we are working in, finishes the proof.

\subsection{Boundedness}
For any fixed $T\le T_\varepsilon$, we set
  \begin{multline}
    \label{Mk}
    M_k = \|\partial Z^{\le 60} u_k\|_{L^\infty L^2}
+ \|Z^{\le 60} u_k\|_{LE^1}
+ \|\good(Z^{\le 60} u_k)\|_{L^2L^2} \\+ \|r^{-1}
    Z^{\le 60} u_k\|_{L^2 L^2} 
+\|\la r\ra^{\frac{1}{2}} \good Z^{\le 60} u_k\|_{L^\infty L^2}
+\|r^{-\frac{1}{2}} Z^{\le 60} u_k\|_{L^\infty L^2}
\\+ \sup_U \|r^{-\frac{1}{2}} \la t-r\ra^{-\frac{1}{2}}
(\partial_t+\partial_r)(r Z^{\le 60} u_k)\|_{L^2L^2(X_U)}
+ \sup_U \|r^{-\frac{1}{2}} (\log \la r\ra)^{-1} \la t-r\ra^{-\frac{1}{2}} Z^{\le 60}u_k\|_{L^2L^2(X_U)}
\\+ \Bigl(\sum_{\tau\le T}\sum_{R\le \tau/2} \|
\partial Z^{\le 50} u_k\|^2_{L^2L^2(\tilde{C}^R_\tau)}\Bigr)^{\frac{1}{2}} 
+ \Bigl[\sum_{\tau\le T} \sum_{R\le \tau/2} \Bigl(R\|\good
\partial Z^{\le 40} u_k\|_{L^2L^2(\tilde{C}^R_\tau)}\Bigr)^2\Bigr]^{\frac{1}{2}}
\\+
\sup_U\Bigl[\sum_{\tau\ge 4U} \Bigl(\frac{U^{\frac{1}{2}}}{\tau^{\frac{1}{2}}\log\la\tau\ra} 
\|\partial Z^{\le 50}
u_k\|_{L^2L^2(\tilde{C}^U_\tau)}\Bigr)^2\Bigr]^{\frac{1}{2}}
+ \sup_U \Bigl[\sum_{\tau\ge 4U}
\Bigl(\frac{U^{\frac{1}{2}}\tau^{\frac{1}{2}}}{\log\la \tau\ra}\| \good \partial
Z^{\le 40} u_k\|_{L^2L^2(\tilde{C}^U_\tau)}\Bigr)^2\Bigr]^{\frac{1}{2}}. 
  \end{multline}
We call these terms $I_k$, $II_k$, \dots, $XI_k$, $XII_k$
respectively.  We shall argue inductively to show that
\begin{equation}
  \label{MkGoal}
  M_k\le 2C_0\varepsilon
\end{equation}
for a uniform constant $C_0$ provided that $T\le T_\varepsilon$.
Indeed, for a universal constant $C_0$, we shall show that
\begin{multline}
  \label{MkGoal2}
  M_k^2 \le (C_0\varepsilon)^2 +C (\log\la T\ra)^3 M_{k-1}^2 M_k +
  C(\log\la T\ra)^3 M_{k-1}M_k^2 + C(\log\la T\ra)^2 M_{k-1}^4 \\+ C
  (\log\la T\ra)^2M_{k-1}^2M_k^2+ C(\log\la T\ra)^5 M_{k-1}^4 +
 C (\log\la T\ra)^5 M_{k-1}^3M_k.
\end{multline}
From this, it follows that $M_1\le C_0\varepsilon$.  Then by the
inductive hypothesis
and \eqref{lifespan}, provided that $c$ and $\varepsilon$ are
sufficiently small (compared to $C_0$), we obtain \eqref{MkGoal}.

We briefly summarize the proof of \eqref{MkGoal2} that is to follow.
Note that terms $I_k$ and $II_k$ are bounded using the energy and
integrated local energy estimate \eqref{le_pert}, while terms
$III_k,\dots, VIII_k$ represent the left side of \eqref{newLE}.  These
eight terms are the principal portions.  To prove \eqref{MkGoal2},
upon applying \eqref{le_pert} and \eqref{newLE}, the product rule will
guarantee that one factor from each nonlinear term will be lower order
(in terms of the number of vector fields).  As this
factor can afford additional vector fields, we may apply our decay
estimates \eqref{crt}, \eqref{cut}, \eqref{ctt}, or
\eqref{weighted_Sobolev} to it.

Closing the argument requires that we obtain additional decay from the
derivative that must be present on at least one factor of every
nonlinear term.  When this derivative is a ``good'' derivative
$\good$, this is relatively simple as the $r^p$-weight allows it to be
bounded with a larger weight, which effectively provides additional
decay to be used for the other factors.  For the $\partial_t-\partial_r$
directions, provided that the factor can admit an additional vector
field \eqref{derivative_decay} yields additional decay.  Here the
decay is in $t-r$, and the use of the ghost weight allows our
estimates to take advantage of this.  Terms $IX_k, \dots, XII_k$ of
\eqref{Mk} are commonly occurring factors where such a procedure is
utilized.

The resulting worst nonlinear term is when $u_{k-1}
(\partial_t-\partial_r) Z^{\le 60} u_{k-1}$ occurs within the right
side of \eqref{newLE}.  Here one integrates $(\partial_t-\partial_r)$
by parts.  When it lands on the lower order factor, the procedure
based in \eqref{derivative_decay} described above can be used.  When
it instead lands on the multiplier term, up to better behaved terms,
$\Box u_k$ is reproduced.  This term is replaced using the nonlinear
equation, and quartic interactions result.  The majority of the terms
can be handled as above and the worst case is again the $u_{k-1}
(\partial_t-\partial_r) Z^{\le 60} u_{k-1}$ interactions.  At this
point, however, the two high order factors can be combined using the
chain rule $w \partial w = \frac{1}{2}\partial w^2$, and integration
by parts can be used to move the derivative to the lower order factors
where \eqref{derivative_decay} can once again be applied.

We note that the extra logarithmic factor in term $VIII_k$ is largely
responsible for our lifespan being $\exp(c/\varepsilon^{\frac{1}{3}})$
rather than $\exp(c/\varepsilon)$ as is known for scalar equations.


In our applications of \eqref{le_pert} and \eqref{newLE}, we set
\begin{equation}\label{h}h^{I,\alpha\beta}_J = A^{I,\alpha\beta}_{JK} u_{k-1}^K +
B^{I,\alpha\beta\gamma}_{JK}\partial_\gamma u_{k-1}^K.
\end{equation}

We proceed with establishing the necessary bound for each of
$I_k,\dots, XII_k$.

\medskip
\noindent $\mathbf{[I_k+II_k]}:$
We begin by showing
\begin{equation}
  \label{I+II_goal}
  I_k^2 + II_k^2 \le (C_0 \varepsilon)^2 +(\log\la T\ra)^{\frac{5}{2}} M_{k-1} M^2_k + (\log\la T\ra)^{\frac{5}{2}}
  M^2_{k-1} M_k.
\end{equation}

Using \eqref{le_pert} and \eqref{smallness}, we
have that
\begin{multline}\label{I+II_1}I_k^2+II_k^2 \le C_0^2\varepsilon^2 + C\int_0^T\int |\Box_h Z^{\le 60}
  u_k|\Bigl(|\partial Z^{\le 60} u_k| + \frac{|Z^{\le 60}
    u_k|}{r}\Bigr)\,dx\,dt
  \\+ C\int_0^T\int \Bigl(|\partial \partial^{\le 1} u_{k-1}| +
  \frac{|\partial^{\le 1} u_{k-1}|}{\la r\ra}\Bigr) |\partial Z^{\le
    60} u_k|\Bigl(|\partial Z^{\le 60} u_k| + \frac{|Z^{\le 60}
    u_k|}{r}\Bigr)\,dx\,dt.
\end{multline}

We first note that we may apply \eqref{weighted_Sobolev} and the
finite speed of propagation to see that
the last term in the right of \eqref{I+II_1} is bounded by
\[ \sum_{0\le j\lesssim \log\la T\ra}  2^{-j}\|r^{-1}(r\partial)^{\le 1} Z^{\le 3}
  u_{k-1}\|_{L^\infty L^2}\int_0^T\int_{A_{2^j}}|\partial Z^{\le 60}
  u_k|\Bigl(|\partial Z^{\le 60} u_k|+\frac{|Z^{\le 60}
    u_k|}{r}\Bigr)\,dx\,dt.\]
The Schwarz inequality and a Hardy inequality then give that
\begin{equation}\label{I.1}\begin{split}
\int_0^T\int \Bigl(|\partial \partial^{\le 1} u_{k-1}| +&
  \frac{|\partial^{\le 1} u_{k-1}|}{\la r\ra}\Bigr) |\partial Z^{\le
    60} u_k|\Bigl(|\partial Z^{\le 60} u_k| + \frac{|Z^{\le 60}
    u_k|}{r}\Bigr)\,dx\,dt
\\&\lesssim \|\partial Z^{\le 3}u_{k-1}\|_{L^\infty L^2}
  \Bigl(\log\la T\ra \|Z^{\le 60} u_k\|^2_{LE^1} \Bigr)
\\
&\lesssim I_{k-1}\Bigl(\log \la T\ra II_k^2 \Bigr), 
\end{split}
\end{equation}
which is controlled by the right side of \eqref{I+II_goal}.

To address the second term in the right side of \eqref{I+II_1}, we
note that
\begin{equation}\label{product1}  |\Box_h Z^{\le 60} u_k| \lesssim \Bigl(|\partial Z^{\le 30} u_{k-1}|
  + |\partial Z^{\le 31} u_k|\Bigr) |\partial^{\le 1} Z^{\le 60}
  u_{k-1}|
  + |Z^{\le 31} u_{k-1}|\Bigl(|\partial Z^{\le 60} u_{k-1}| +
  |\partial Z^{\le 60} u_k|\Bigr).
\end{equation}
We first write
\begin{multline}\label{I.a}
  \int_0^T\int |\Box_h Z^{\le 60}
  u_k|\Bigl(|\partial Z^{\le 60} u_k| + \frac{|Z^{\le 60}
  u_k|}{r}\Bigr)\,dx\,dt
\\\lesssim \sum_{\tau\le T}
                               \Bigl(\sum_{R\le \tau/2}  \int\int_{C^R_\tau} |\Box_h Z^{\le 60}
  u_k|\Bigl(|\partial Z^{\le 60} u_k| + \frac{|Z^{\le 60}
    u_k|}{r}\Bigr)\,dx\,dt   \\ + \sum_{U\le
                               \tau/4}   \int
                               \int_{C^U_\tau} |\Box_h Z^{\le 60}
  u_k|\Bigl(|\partial Z^{\le 60} u_k| + \frac{|Z^{\le 60}
    u_k|}{r}\Bigr)\,dx\,dt\Bigr).
\end{multline}

By \eqref{crt} and \eqref{ctt}, we have
\begin{multline*}
\int \int_{C^R_\tau} \Bigl(|\partial Z^{\le 30} u_{k-1}|
  + |\partial Z^{\le 31} u_k|\Bigr) |\partial^{\le 1} Z^{\le 60}
  u_{k-1}| \Bigl(|\partial Z^{\le 60} u_k| + \frac{|Z^{\le 60}
    u_k|}{r}\Bigr)\,dx\,dt
\\  \lesssim \Bigl(\|\partial Z^{\le 34}
u_{k-1}\|_{L^2L^2(\tilde{C}^R_\tau)} + \|\partial Z^{\le 35}
u_k\|_{L^2L^2(\tilde{C}^R_\tau)}
+R\|\good \partial Z^{\le 33}
u_{k-1}\|_{L^2L^2(\tilde{C}^R_\tau)} + R\|\good \partial Z^{\le 34}
u_k\|_{L^2L^2(\tilde{C}^R_\tau)}
\Bigr)
\\\times \|\la r\ra^{-1}\partial^{\le 1} Z^{\le 60}
u_{k-1}\|_{L^2L^2(C^R_\tau)}  \|\la r\ra^{-1}(\partial
Z^{\le 60} u_k, r^{-1}Z^{\le 60}u_k)\|_{L^2L^2(C^R_\tau)}.
\end{multline*}
Noting, for example, that
\[\|\la r\ra^{-1} \partial Z^{\le 60} u\|^2_{L^2L^2}
= \sum_j 2^{-j} \Bigl(2^{-j} \|\partial Z^{\le 60}
  u\|^2_{L^2L^2([0,T]\times A_{2^j})}\Bigr) \lesssim \|Z^{\le 60} u\|^2_{LE^1},
\]
we thus have
\begin{multline}\label{I.b}
\sum_{\tau\le T} \sum_{R\le \tau/2}\int \int_{C^R_\tau} \Bigl(|\partial Z^{\le 30} u_{k-1}|
  + |\partial Z^{\le 31} u_k|\Bigr) |\partial^{\le 1} Z^{\le 60}
  u_{k-1}| \Bigl(|\partial Z^{\le 60} u_k| + \frac{|Z^{\le 60}
    u_k|}{r}\Bigr)\,dx\,dt
\\\lesssim \Bigl(IX_{k-1}+IX_k+X_{k-1}+X_k\Bigr)\Bigl(IV_{k-1}+II_{k-1}\Bigr)II_k.
\end{multline}
Similarly,
\begin{multline*}
  \int\int_{C^R_\tau} |Z^{\le 31} u_{k-1}|\Bigl(|\partial Z^{\le 60}
  u_{k-1}|+|\partial Z^{\le 60} u_k|\Bigr) \Bigl(|\partial Z^{\le 60}
  u_k| + \frac{|Z^{\le 60} u_k|}{r}\Bigr)\,dx\,dt
\\\lesssim \Bigl(\|r^{-1} Z^{\le 35} u_{k-1}\|_{L^2L^2(\tilde{C}^R_\tau)} + \|\good
Z^{\le 34} u_{k-1}\|_{L^2L^2(\tilde{C}^R_\tau)}\Bigr) \|\la
r\ra^{-\frac{1}{2}} (\partial Z^{\le 60} u_k, r^{-1}Z^{\le
  60}u_k)\|_{L^2L^2(C^R_\tau)} \\\times \Bigl(\|\la 
r\ra^{-\frac{1}{2}} \partial Z^{\le 60} u_{k-1}\|_{L^2L^2(C^R_\tau)} + \|\la
r\ra^{-\frac{1}{2}} \partial Z^{\le 60} u_k\|_{L^2L^2(C^R_\tau)}\Bigr). 
\end{multline*}
Since the initial data are supported in $\{|x|\le 2\}$, we get
\begin{equation}\label{L2L2toLE}
\|\la r\ra^{-\frac{1}{2}} \partial Z^{\le 60} u\|^2_{L^2L^2}
\le \sum_{j\lesssim \log\la T\ra} 2^{-j} \|\partial Z^{\le 60}
u\|^2_{L^2L^2([0,T]\times A_{2^j})} 
\lesssim \log\la T\ra \|Z^{\le 60} u\|_{LE^1}^2.
\end{equation}
It follows that
\begin{multline}\label{I.c}
  \sum_{\tau\le T} \sum_{R\le \tau/2}\int\int_{C^R_\tau} |Z^{\le 31} u_{k-1}|\Bigl(|\partial Z^{\le 60}
  u_{k-1}|+|\partial Z^{\le 60} u_k|\Bigr) \Bigl(|\partial Z^{\le 60}
  u_k| + \frac{|Z^{\le 60} u_k|}{r}\Bigr)\,dx\,dt
\\\lesssim \Bigl(IV_{k-1}+III_{k-1}\Bigr)\cdot II_k\cdot
(\log\la T\ra)^{\frac{1}{2}} \Bigl(II_{k-1}+II_k\Bigr).
\end{multline}

By \eqref{cut},
\begin{multline*}
\int \int_{C^U_\tau} \Bigl(|\partial Z^{\le 30} u_{k-1}|
  + |\partial Z^{\le 31} u_k|\Bigr) |\partial^{\le 1} Z^{\le 60}
  u_{k-1}| \Bigl(|\partial Z^{\le 60} u_k| + \frac{|Z^{\le 60}
    u_k|}{r}\Bigr)\,dx\,dt
\\  \lesssim \frac{U^{\frac{1}{2}}}{\tau^{\frac{1}{2}}}\Bigl(\|\partial Z^{\le 34}
u_{k-1}\|_{L^2L^2(\tilde{C}^U_\tau)} + \|\partial Z^{\le 35}
u_k\|_{L^2L^2(\tilde{C}^U_\tau)}
+\tau \|\good \partial Z^{\le 33}
u_{k-1}\|_{L^2L^2(\tilde{C}^U_\tau)} + \tau \|\good \partial Z^{\le 34}
u_k\|_{L^2L^2(\tilde{C}^U_\tau)}
\Bigr)
\\\times U^{-1}\|\la r\ra^{-\frac{1}{2}}\partial^{\le 1} Z^{\le 60}
u_{k-1}\|_{L^2L^2(C^U_\tau)}  \|\la r\ra^{-\frac{1}{2}}(\partial
Z^{\le 60} u_k, r^{-1}Z^{\le 60}u_k)\|_{L^2L^2(C^U_\tau)}.
\end{multline*}
Thus, using \eqref{L2L2toLE},
\begin{multline}\label{I.d}
  \sum_{\tau\le T} \sum_{U\le \tau/4}\int \int_{C^U_\tau} \Bigl(|\partial Z^{\le 30} u_{k-1}|
  + |\partial Z^{\le 31} u_k|\Bigr) |\partial^{\le 1} Z^{\le 60}
  u_{k-1}| \Bigl(|\partial Z^{\le 60} u_k| + \frac{|Z^{\le 60}
    u_k|}{r}\Bigr)\,dx\,dt
\\\lesssim \log\la T\ra \Bigl(XI_{k-1} + XI_k + XII_{k-1} +
XII_k\Bigr)\Bigl(\log\la T\ra  VIII_{k-1} +(\log \la T\ra)^{\frac{1}{2}}II_{k-1}\Bigr) \Bigl((\log\la T\ra)^{\frac{1}{2}} II_k\Bigr).
\end{multline}

Relying on \eqref{cut} again, we have
\begin{multline*}
  \int \int_{C^U_\tau} |Z^{\le 31} u_{k-1}| \Bigl(|\partial Z^{\le 60}
  u_{k-1}| + |\partial Z^{\le 60} u_k|\Bigr)\Bigl(|\partial Z^{\le 60} u_k| + \frac{|Z^{\le 60}
    u_k|}{r}\Bigr)\,dx\,dt
\\\lesssim \frac{1}{U^{\frac{1}{2}} \tau^{\frac{1}{2}}}\Bigl(\|Z^{\le
  35}u_{k-1}\|_{L^2L^2(\tilde{C}^U_\tau)} +
\|(\partial_t+\partial_r)(r Z^{\le 34}
u_{k-1})\|_{L^2L^2(\tilde{C}^U_\tau)}\Bigr)\\\times
\Bigl(\|\la r\ra^{-\frac{1}{2}}\partial
Z^{\le 60} u_{k-1}\|_{L^2L^2(C^U_\tau)} + \|\la r\ra^{-\frac{1}{2}}\partial
Z^{\le 60} u_k\|_{L^2L^2(C^U_\tau)}\Bigr) \| \la r\ra^{-\frac{1}{2}}(\partial Z^{\le 60} u_k,
r^{-1} Z^{\le 60} u_k)\|_{L^2L^2(C^U_\tau)}.
\end{multline*}
Upon using \eqref{L2L2toLE}, it follows that
\begin{multline}\label{I.e}
\sum_{\tau\le T} \sum_{U\le \tau/4}  \int \int_{C^U_\tau} |Z^{\le 31} u_{k-1}| \Bigl(|\partial Z^{\le 60}
  u_{k-1}| + |\partial Z^{\le 60} u_k|\Bigr)\Bigl(|\partial Z^{\le 60} u_k| + \frac{|Z^{\le 60}
    u_k|}{|x|}\Bigr)\,dx\,dt \\\lesssim
\Bigl(\log \la T\ra VIII_{k-1}+VII_{k-1}\Bigr) \Bigl((\log\la
T\ra)^{\frac{1}{2}}II_{k-1}+ (\log\la T\ra)^{\frac{1}{2}}II_k\Bigr)\Bigl((\log \la T\ra)^{\frac{1}{2}}II_k\Bigr).
\end{multline}

By plugging bounds \eqref{I.1}, \eqref{I.b}, \eqref{I.c}, \eqref{I.d},
\eqref{I.e} into \eqref{I+II_1} and\eqref{I.a}
we obtain the desired bound \eqref{I+II_goal}.

\medskip
\noindent$\mathbf{[III_k+IV_k+V_k+VI_k+VII_k+VIII_k]}$:
Here, relying on \eqref{newLE}, we show that
\begin{multline}
  \label{III_goal}
  III^2_k+IV^2_k+V^2_k+VI^2_k+VII^2_k+VIII^2_k \le (C_0\varepsilon)^2 + C (\log\la T\ra)^3
  M_{k-1}M^2_k + C (\log\la T\ra)^3 M^2_{k-1} M_k
\\+ C(\log \la T\ra)^5 M_{k-1}^4 + C(\log\la T\ra)^5 M_{k-1}^3 M_k.
\end{multline}
The first term in the right of \eqref{newLE} is bounded by
$C_0^2\varepsilon^2$ due to \eqref{smallness}.  We will proceed, in order, to
showing that each of the terms, other than the $\Box_h Z^{\le 60} u_k$
term, in the right side of
\eqref{newLE} are bounded by
\begin{equation}\label{III_goal2}C\log\la
  T\ra^3\Bigl(M_{k-1}^2M_k+M_{k-1}M_k^2\Bigr).
\end{equation}
We will argue separately that the $\Box_h Z^{\le 60} u_k$ term is
bounded by the right side of \eqref{III_goal}, which will establish
the desired bound.

To control the second term in the right side of \eqref{newLE}, we will
consider the integral at an arbitrary $t\in [0,T]$, and we fix a
dyadic value $\tau$ so that $t\in [\tau, 2\tau]$.
 For $1\le R\le \tau/2$, we can apply \eqref{crt} and \eqref{ctt} and a
Hardy inequality (after expanding the range of integration of the norm
of $r^{-1} |Z^{\le 60} u_k|$ from $A_R$ to $\R^3$) to see
\begin{multline*}
  \int_{A_R}  r |\partial^{\le 1} u_{k-1}| |\partial Z^{\le 60} u_k| \Bigl(|\partial
  Z^{\le 60} u_k| + \frac{|Z^{\le 60} u_k|}{r}\Bigr)\,dx
\\  \lesssim \Bigl(R^{-1} \|Z^{\le 5}
  u_{k-1}\|_{L^2L^2(\tilde{C}^R_\tau)} + \|\good Z^{\le
    4}u_{k-1}\|_{L^2L^2(\tilde{C}^R_\tau)}\Bigr) \|\partial Z^{\le 60}
  u_k(t,\cd)\|_{L^2(A_R)} \|\partial Z^{\le 60} u_k\|_{L^\infty L^2}.
\end{multline*}
And hence, using the Schwarz inequality,
\[
  \sup_{t\in [0,T]} \sum_{R\le \tau/2}  \int_{A_R}  r |\partial^{\le 1} u_{k-1}| |\partial Z^{\le 60} u_k| \Bigl(|\partial
  Z^{\le 60} u_k| + \frac{|Z^{\le 60} u_k|}{r}\Bigr)\,dx \lesssim
  \Bigl(IV_{k-1}+III_{k-1}\Bigr) I_k^2,\]
which is dominated by \eqref{III_goal2}.  For the remainder of this
term, \eqref{cut} and a Hardy inequality show that
\begin{multline*}
 \sum_{1\le U\le \tau/4} \int_{\la t-r\ra \approx U}  r |\partial^{\le 1} u_{k-1}| |\partial Z^{\le 60} u_k| \Bigl(|\partial
  Z^{\le 60} u_k| + \frac{|Z^{\le 60} u_k|}{r}\Bigr)\,dx
\\\lesssim \sup_U\Bigl(\frac{1}{\tau^{\frac{1}{2}}U^{\frac{1}{2}}} \|Z^{\le
  4} u_{k-1}\|_{L^2L^2(\tilde{C}^U_\tau)} +
\frac{1}{\tau^{\frac{1}{2}} U^{\frac{1}{2}}}
  \|(\partial_t+\partial_r)(rZ^{\le 3}
  u_{k-1})\|_{L^2L^2(\tilde{C}^U_\tau)}\Bigr) \|\partial Z^{\le 60} u_k(t,\cd)\|^2_{L^2}.
\end{multline*}
The supremum of this is bounded by
\[\Bigl(\log\la T\ra VIII_{k-1} + VII_{k-1}\Bigr) I_k^2.\]
Thus,
\begin{multline}
  \label{III.tbdy}
  \sup_{t\in [0,T]} \int r |\partial^{\le 1} u_{k-1}| |\partial Z^{\le 60} u_k| \Bigl(|\partial
  Z^{\le 60} u_k| + \frac{|Z^{\le 60} u_k|}{r}\Bigr)\,dx
\\\lesssim \Bigl(IV_{k-1}+III_{k-1} + \log\la T\ra VIII_{k-1} + VII_{k-1}\Bigr)I_k^2.
\end{multline}

We proceed to showing that
\begin{multline} \label{IIIBox}\sup_{U\ge 1}\sup_{t\in [0,T]} \Bigl|\int_0^t \int  e^{-\sigma_U(t-r)}
  \Box_h Z^{\le 60} u_k\cdot\Bigl(\partial_t+\partial_r\Bigr)(rZ^{\le
    60}u_k)\,dx\,dt\Bigr|
\\\lesssim (\log\la T\ra)^3 \Bigl(M_{k-1}^2M_k + M_{k-1}M_k^2\Bigr)
+ (\log \la T\ra)^5 M_{k-1}^4 + (\log \la T\ra)^5 M_{k-1}^3M_k.
\end{multline}

\begin{proof}[Proof of \eqref{IIIBox}] 
The most delicate terms in this analysis are those of the form
$u_{k-1}(\partial_t-\partial_r)Z^{\le 60}u_{k-1}$.  
Here we have a bad derivative occurring at the highest regularity, and
thus there is not room to apply, for example, \eqref{derivative_decay}
directly in order to get additional decay.

We begin by examining the
other terms.  To this end, we set $\omega = (1,-x/r)$ and note that
\begin{multline}\label{Box.minus.udu}
  |\Box_h Z^{\le 60} u_k^I - a^{I,\alpha}_{JK}\omega_\alpha u_{k-1}^J
  (\partial_t-\partial_r) Z^{\le 60} u_{k-1}^K| \lesssim
  |Z^{\le 30} u_{k-1}| |\good Z^{\le 60} u_{k-1}|
  +|\partial Z^{\le 30} u_{k-1}| |\partial^{\le 1} Z^{\le 60} u_{k-1}|
 \\ +|Z^{\le 30} u_{k-1}||\partial Z^{\le 59} u_{k-1}|+|\partial Z^{\le 31} u_k| |\partial^{\le 1} Z^{\le 60} u_{k-1}| 
  + |\partial^{\le 1} Z^{\le 30} u_{k-1}| |\partial^2 Z^{\le 59} u_k|.
\end{multline}

Using \eqref{crt}, \eqref{cut}, and \eqref{ctt} gives us that
\begin{multline*}
  \int\int_{C^R_\tau} r |Z^{\le 30} u_{k-1}| |\good Z^{\le 60}
  u_{k-1}| \Bigl(|\good Z^{\le 60} u_k| + r^{-1}
  |Z^{\le 60} u_k|\Bigr)\,dx\,dt
  \\\lesssim \Bigl(R^{-1}\|Z^{\le 34}
  u_{k-1}\|_{L^2L^2(\tilde{C}^R_\tau)} + \|\good Z^{\le 33}
  u_{k-1}\|_{L^2L^2(\tilde{C}^R_\tau)}\Bigr) \|\good Z^{\le 60}
  u_{k-1}\|_{L^2L^2(C^R_\tau)} \\\times \Bigl(\|\good Z^{\le 60}
  u_k\|_{L^2L^2(C^R_\tau)} + \|r^{-1}Z^{\le 60} u_k\|_{L^2L^2(C^R_\tau)}\Bigr),
\end{multline*}
and respectively
\begin{multline*}
  \int\int_{C^U_\tau} r |Z^{\le 30} u_{k-1}| |\good Z^{\le 60}
  u_{k-1}| \Bigl(|\good Z^{\le 60} u_k| + r^{-1}
  |Z^{\le 60} u_k|\Bigr) \,dx\,dt
  \\\lesssim \Bigl(\frac{1}{U^{\frac{1}{2}}\tau^{\frac{1}{2}}}\|Z^{\le 34}
  u_{k-1}\|_{L^2L^2(\tilde{C}^U_\tau)} + \frac{1}{U^{\frac{1}{2}}\tau^{\frac{1}{2}}}\|(\partial_t+\partial_r)(r Z^{\le 33}
  u_{k-1})\|_{L^2L^2(\tilde{C}^U_\tau)}\Bigr) \|\good Z^{\le 60}
  u_{k-1}\|_{L^2L^2(C^U_\tau)} \\\times  \Bigl(\|\good Z^{\le 60}
  u_k\|_{L^2L^2(C^U_\tau)} + \|r^{-1}Z^{\le 60} u_k\|_{L^2L^2(C^U_\tau)}\Bigr).
\end{multline*}
Upon summing over $R\le \tau/2$, $U\le \tau/4$ and $\tau \le T$, we get
\begin{multline}
  \label{IIIb1}
  \int_0^T\int r |Z^{\le 30} u_{k-1}||\good Z^{\le 60} u_{k-1}|
  \Bigl(|\good Z^{\le 60} u_k| + r^{-1}
  |Z^{\le 60} u_k|\Bigr)\,dx\,dt
  \\\lesssim (IV_{k-1}+III_{k-1}+\log\la T\ra VIII_{k-1} + VII_{k-1}) III_{k-1} (III_k+IV_k).
\end{multline}

Another application of \eqref{crt} and \eqref{ctt} gives
\begin{multline*}
  \int\int_{C^R_\tau} r |\partial Z^{\le 30} u_{k-1}| |\partial^{\le 1} Z^{\le
    60} u_{k-1}| \Bigl(|\good Z^{\le 60} u_k| + r^{-1}|Z^{\le 60}
  u_k|\Bigr)\,dx\,dt
\\  \lesssim \Bigl(\|\partial Z^{\le 34}
  u_{k-1}\|_{L^2L^2(\tilde{C}^R_\tau)} + R \|\good \partial Z^{\le 33}
  u_{k-1}\|_{L^2L^2(\tilde{C}^R_\tau)}\Bigr) R^{-1} \|\partial^{\le 1} Z^{\le
    60} u_{k-1}\|_{L^2L^2(C^R_\tau)} \\\times \Bigl(\|\good Z^{\le 60}
  u_k\|_{L^2L^2(C^R_\tau)} + \|r^{-1} Z^{\le 60} u_k\|_{L^2L^2(C^R_\tau)}\Bigr).
\end{multline*}
Similarly, using \eqref{cut}, we get
\begin{multline*}
  \int\int_{C^U_\tau}  |\partial Z^{\le 30} u_{k-1}| |\partial^{\le 1} Z^{\le
    60} u_{k-1}| \Bigl(|(\partial_t+\partial_r)(r Z^{\le 60} u_k)|\Bigr)\,dx\,dt
\\  \lesssim \Bigl(\frac{U^{\frac{1}{2}}}{\tau^{\frac{1}{2}}}\|\partial Z^{\le 34}
  u_{k-1}\|_{L^2L^2(\tilde{C}^U_\tau)} + U^{\frac{1}{2}}\tau^{\frac{1}{2}}\|(\partial_t+\partial_r) \partial Z^{\le 33}
  u_{k-1}\|_{L^2L^2(\tilde{C}^U_\tau)}\Bigr)
  \frac{1}{U^{\frac{1}{2}}\tau^{\frac{1}{2}}}\|\partial^{\le 1} Z^{\le
    60} u_{k-1}\|_{L^2L^2(C^U_\tau)} \\\times \frac{1}{U^{\frac{1}{2}}\tau^{\frac{1}{2}}}\|(\partial_t+\partial_r)(r Z^{\le 60}
  u_k)\|_{L^2L^2(C^U_\tau)}.
\end{multline*}
Upon summing, these give
\begin{multline}
  \label{IIIb2}
\int_0^T\int |\partial Z^{\le 30} u_{k-1}| |\partial^{\le 1} Z^{\le
  60} u_{k-1}| |(\partial_t+\partial_r)(rZ^{\le 60}u_k)|\,dx\,dt
\\\lesssim (IX_{k-1}+X_{k-1})(II_{k-1} + IV_{k-1})(III_k+IV_k) \\+
(\log\la T\ra)^2\Bigl(XI_{k-1}+XII_{k-1}\Bigr)\Bigl((\log\la
T\ra)^{\frac{1}{2}} II_{k-1} + \log\la T\ra VIII_{k-1}\Bigr) VII_k.
\end{multline}
Following the same argument, we also obtain
\begin{multline}
  \label{IIIb3}
\int_0^T\int |\partial Z^{\le 31} u_{k}| |\partial^{\le 1} Z^{\le
  60} u_{k-1}| |(\partial_t+\partial_r)(rZ^{\le 60}u_k)|\,dx\,dt
\\\lesssim (IX_{k}+X_{k})(II_{k-1} + IV_{k-1})(III_k+IV_k) \\+
(\log\la T\ra)^2\Bigl(XI_{k}+XII_{k}\Bigr)\Bigl((\log\la
T\ra)^{\frac{1}{2}} II_{k-1} + \log\la T\ra VIII_{k-1}\Bigr) VII_k.
\end{multline}

We now apply both \eqref{derivative_decay} and \eqref{crt} (and \eqref{ctt}) to see that
\begin{multline*}
  \int\int_{C^R_\tau} |Z^{\le 30} u_{k-1}| |\partial Z^{\le 59}
  u_{k-1}| |(\partial_t+\partial_r)(rZ^{\le 60} u_k)|\,dx\,dt
  \\\lesssim \Bigl(\frac{1}{R} \|Z^{\le 34}
    u_{k-1}\|_{L^2L^2(\tilde{C}^R_\tau)} + \|\good Z^{\le 33}
    u_{k-1}\|_{L^2L^2(\tilde{C}^R_\tau)}\Bigr)
   \Bigl(\frac{1}{R} \|Z^{\le 60} u_{k-1}\|_{L^2L^2(C^R_\tau)} +
   \|\good Z^{\le 59} u_{k-1}\|_{L^2L^2(C^R_\tau)}\Bigr) \\\times\Bigl( \|\good
   Z^{\le 60} u_k\|_{L^2L^2(C^R_\tau)} + \frac{1}{R}\| Z^{\le 60} u_k\|_{L^2L^2(C^R_\tau)}\Bigr).
\end{multline*}
While \eqref{derivative_decay} and \eqref{cut} give
  \begin{multline*}
    \int\int_{C^U_\tau} |Z^{\le 30} u_{k-1}||\partial Z^{\le 59} u_{k-1}|
    |(\partial_t+\partial_r)(r Z^{\le 60} u_k)|\,dx\,dt
    \\\lesssim \frac{1}{U^{\frac{1}{2}}\tau^{\frac{1}{2}}} \Bigl(\|Z^{\le 34}
    u_{k-1}\|_{L^2L^2(\tilde{C}^U_\tau)} +  \|(\partial_t+\partial_r)
    (r Z^{\le 33}
    u_{k-1})\|_{L^2L^2(\tilde{C}^U_\tau)}\Bigr)\\\times
   \frac{1}{U^{\frac{1}{2}}\tau^{\frac{1}{2}}}  \Bigl(\|Z^{\le 60} u_{k-1}\|_{L^2L^2(C^U_\tau)} +
  \|(\partial_t+\partial_r)(r Z^{\le 60}
   u_{k-1})\|_{L^2L^2(C^U_\tau)}\Bigr) \\\times \frac{1}{U^{\frac{1}{2}}\tau^{\frac{1}{2}}}
  \|(\partial_t+\partial_r)(r 
   Z^{\le 60} u_k)\|_{L^2L^2(C^U_\tau)}.
 \end{multline*}
Upon summation, this gives
\begin{multline}
  \label{IIIb4a}
  \int_0^T\int |Z^{\le 30} u_{k-1}| |\partial
  Z^{\le 59} u_{k-1}| |(\partial_t+\partial_r)(rZ^{\le 60} u_k)|\,dx\,dt
\\\lesssim (IV_{k-1} + III_{k-1})^2(III_k+IV_k)
\\+ \log\la T\ra \Bigl(\log \la T\ra VIII_{k-1} +
VII_{k-1}\Bigr)^2 VII_k.
\end{multline}
Very similar arguments give
\begin{multline*}
    \int\int_{C^R_\tau} |\partial^{\le 1} Z^{\le 30} u_{k-1}||\partial^2 Z^{\le 59} u_k|
    |(\partial_t+\partial_r)(r Z^{\le 60} u_k)|\,dx\,dt
    \\\lesssim \Bigl(\frac{1}{R} \|Z^{\le 35}
    u_{k-1}\|_{L^2L^2(\tilde{C}^R_\tau)} + \|\good Z^{\le 34}
    u_{k-1}\|_{L^2L^2(\tilde{C}^R_\tau)}\Bigr)
   \Bigl(\frac{1}{R} \|\partial Z^{\le 60} u_k\|_{L^2L^2(C^R_\tau)} +
   \|\good Z^{\le 60} u_k\|_{L^2L^2(C^R_\tau)}\Bigr) \\\times\Bigl( \|\good
   Z^{\le 60} u_k\|_{L^2L^2(C^R_\tau)} + \frac{1}{R}\| Z^{\le 60} u_k\|_{L^2L^2(C^R_\tau)}\Bigr)
  \end{multline*}
and
  \begin{multline*}
    \int\int_{C^U_\tau} |\partial^{\le 1} Z^{\le 30} u_{k-1}||\partial^2 Z^{\le 59} u_k|
    |(\partial_t+\partial_r)(r Z^{\le 60} u_k)|\,dx\,dt
    \\\lesssim \frac{1}{U^{\frac{1}{2}}\tau^{\frac{1}{2}}} \Bigl(\|Z^{\le 35}
    u_{k-1}\|_{L^2L^2(\tilde{C}^U_\tau)} +  \|(\partial_t+\partial_r)
    (r Z^{\le 34}
    u_{k-1})\|_{L^2L^2(\tilde{C}^U_\tau)}\Bigr)\\\times
   \frac{1}{U^{\frac{1}{2}}\tau^{\frac{1}{2}}}  \Bigl(\|\partial Z^{\le 60} u_k\|_{L^2L^2(C^U_\tau)} +
  \|(\partial_t+\partial_r)(r Z^{\le 60}
   u_k)\|_{L^2L^2(C^U_\tau)}\Bigr) \\\times \frac{1}{U^{\frac{1}{2}}\tau^{\frac{1}{2}}}
  \|(\partial_t+\partial_r)(r 
   Z^{\le 60} u_k)\|_{L^2L^2(C^U_\tau)},
 \end{multline*}
which yield
\begin{multline}
  \label{IIIb4}
  \int_0^T\int |\partial^{\le 1} Z^{\le 30} u_{k-1}| |\partial^2
  Z^{\le 59} u_k| |(\partial_t+\partial_r)(rZ^{\le 60} u_k)|\,dx\,dt
\\\lesssim (IV_{k-1} + III_{k-1})(II_k+III_k)(III_k+IV_k)
\\+ \Bigl(\log \la T\ra VIII_{k-1} +
VII_{k-1}\Bigr)\Bigl(II_k + \log\la T\ra VII_k\Bigr) VII_k.
\end{multline}

As the right sides of \eqref{IIIb1}, \eqref{IIIb2}, \eqref{IIIb3},
\eqref{IIIb4a}, and
\eqref{IIIb4} are bounded by \eqref{III_goal2}, to complete the bound
\eqref{IIIBox}, we need only examine
\[\sup_{U\ge 1}\sup_{t\in [0,T]}\Bigl|\int_0^t\int  e^{-\sigma_U(t-r)}
  a^{I,\alpha}_{JK}\omega_\alpha u_{k-1}^J (\partial_t-\partial_r)
  Z^{\le 60} u_{k-1}^K (\partial_t+\partial_r)(rZ^{\le 60} u_k^I)\,dx\,dt\Bigr|.\]
The argument that we shall use here is reminiscent of normal forms.

We first integrate by parts to see that
\begin{multline}
  \label{normalformIBP}
  \int_0^t\int e^{-\sigma_U(t-r)}
  a^{I,\alpha}_{JK}\omega_\alpha u_{k-1}^J (\partial_t-\partial_r)
  Z^{\le 60} u_{k-1}^K (\partial_t+\partial_r)(rZ^{\le 60}
  u_k^I)\,dx\,dt
\\=\int e^{-\sigma_U(t-r)} a^{I,\alpha}_{JK}\omega_\alpha
u^J_{k-1} Z^{\le 60} u^K_{k-1} (\partial_t+\partial_r)(rZ^{\le 60}
u^I_k)\,dx\Bigl|_0^t
\\+2\int_0^t\int r^{-1} e^{-\sigma_U(t-r)} a^{I,\alpha}_{JK}
\omega_\alpha u^J_{k-1} Z^{\le 60} u^K_{k-1}
(\partial_t+\partial_r)(rZ^{\le 60} u^I_k)\,dx\,dt
\\+2\int_0^t\int \sigma'_U(t-r)e^{-\sigma_U(t-r)}
a^{I,\alpha}_{JK}\omega_\alpha u^J_{k-1} Z^{\le 60} u^K_{k-1}
(\partial_t+\partial_r)(rZ^{\le 60} u_k^I)\,dx\,dt
\\-\int_0^t\int e^{-\sigma_U(t-r)} a^{I,\alpha}_{JK} \omega_\alpha
(\partial_t-\partial_r)u_{k-1}^J Z^{\le 60} u^K_{k-1}
(\partial_t+\partial_r)(rZ^{\le 60} u_k^I)\,dx\,dt
\\- \int_0^t \int e^{-\sigma_U(t-r)}
a^{I,\alpha}_{JK}\omega_\alpha u^J_{k-1} Z^{\le 60} u_{k-1}^K
(\partial_t^2-\partial_r^2)(rZ^{\le 60} u_k^I)\,dx\,dt.
\end{multline}
We shall proceed through arguments that will bound each term in
\eqref{normalformIBP}.

For $t$ fixed and $\tau \approx t$, we may
apply \eqref{crt}, \eqref{ctt}, and the Schwarz inequality to see that
\begin{multline*}
  \sum_{R\le \tau/2}\int_{A_R} r |u_{k-1}| |Z^{\le 60} u_{k-1}| \Bigl(|\good Z^{\le 60}
  u_k| + r^{-1}|Z^{\le 60} u_k|\Bigr)\,dx
\\\lesssim \Bigl(\|r^{-1} Z^{\le 4} u_{k-1}\|_{L^2L^2} +
\|\good Z^{\le 3} u_{k-1}\|_{L^2L^2}\Bigr) 
\|r^{-\frac{1}{2}}Z^{\le 60} u_{k-1}(t,\cd)\|_{L^2}\\\times
\Bigl(\|\la r\ra^{\frac{1}{2}}\good Z^{\le 60} u_k(t,\cd)\|_{L^2} +
\|r^{-\frac{1}{2}}Z^{\le 60} u_k(t,\cd)\|_{L^2}\Bigr).
\end{multline*}
Using \eqref{cut}, we instead get
\begin{multline*}
  \sum_{U\le \tau/4}\int_{\la t-r\ra\approx U} r |u_{k-1}| |Z^{\le 60} u_{k-1}| \Bigl(|\good Z^{\le 60}
  u_k| + r^{-1}|Z^{\le 60} u_k|\Bigr)\,dx
\\\lesssim \sup_U \Bigl[\frac{1}{U^{\frac{1}{2}}\tau^{\frac{1}{2}}}\Bigl(\|Z^{\le 4} u_{k-1}\|_{L^2L^2(\tC^U_\tau)} +
\|(\partial_t+\partial_r)(r Z^{\le 3} u_{k-1})\|_{L^2L^2(\tC^U_\tau)}\Bigr)\Bigr]
\|r^{-\frac{1}{2}}Z^{\le 60} u_{k-1}(t,\cd)\|_{L^2}\\\times
\Bigl(\|\la r\ra^{\frac{1}{2}} \good Z^{\le 60} u_k(t,\cd)\|_{L^2} +\|r^{-\frac{1}{2}}Z^{\le 60} u_k(t,\cd)\|_{L^2}\Bigr).
\end{multline*}
As such,
\begin{multline}
  \label{nf1}
  \sup_U\sup_{t\in [0,T]}\Bigl|\int e^{-\sigma_U(t-r)} a^{I,\alpha}_{JK}\omega_\alpha
u^J_{k-1} Z^{\le 60} u^K_{k-1} (\partial_t+\partial_r)(rZ^{\le 60}
u^I_k)\,dx\Bigl|_0^t\Bigr|
\\\lesssim \Bigl(IV_{k-1}+III_{k-1}+ \log\la T\ra
VIII_{k-1} + VII_{k-1}\Bigr)VI_{k-1}(V_k+VI_k).
\end{multline}

For the second term in the right of \eqref{normalformIBP}, provided
$R\le \tau/2$, we may apply \eqref{crt} or \eqref{ctt} to see that
\begin{multline}\label{nf2crt}
  \int\int_{C^R_\tau}  |u_{k-1}| |Z^{\le 60} u_{k-1}| \Bigl(|\good
  Z^{\le 60} u_k| + r^{-1} |Z^{\le 60} u_k|\Bigr)\,dx\,dt
\\\lesssim \Bigl(\|r^{-1} Z^{\le 4}
u_{k-1}\|_{L^2L^2(\tC^R_\tau)} + \|\good Z^{\le 3}
u_{k-1}\|_{L^2L^2(\tC^R_\tau)}\Bigr) \|r^{-1} Z^{\le 60}
u_{k-1}\|_{L^2L^2(C^R_\tau)} \\\times\Bigl(\|\good Z^{\le 60}
u_k\|_{L^2L^2(C^R_\tau)} + \|r^{-1} Z^{\le 60} u_k\|_{L^2L^2(C^R_\tau)}\Bigr),
\end{multline}
and for $U\le \tau/4$, \eqref{cut} gives
\begin{multline}\label{nf2cut}
  \int\int_{C^U_\tau}  r^{-1} |u_{k-1}| |Z^{\le 60} u_{k-1}| |(\partial_t+\partial_r)
 (r Z^{\le 60} u_k)|\,dx\,dt
\\\lesssim \Bigl(\frac{1}{\tau}\| Z^{\le 4}
u_{k-1}\|_{L^2L^2(\tC^U_\tau)} + \frac{1}{\tau}\|(\partial_t+\partial_r)(r Z^{\le 3}
u_{k-1})\|_{L^2L^2(\tC^U_\tau)}\Bigr)  \frac{1}{\tau} \|Z^{\le 60}
u_{k-1}\|_{L^2L^2(C^U_\tau)} \\\times \frac{1}{U^{\frac{1}{2}}\tau^{\frac{1}{2}}}\|(\partial_t+\partial_r)(r Z^{\le 60}
u_k)\|_{L^2L^2(C^U_\tau)}.
\end{multline}
Upon summing, this results in
\begin{equation}\label{nf2}
  \int_0^T\int r^{-1} |u_{k-1}| |Z^{\le 60} u_{k-1}|
  |(\partial_t+\partial_r)(rZ^{\le 60} u_k)|\,dx\,dt
\lesssim (IV_{k-1}+III_{k-1}) IV_{k-1} (III_k+IV_k+VII_k),
\end{equation}
which suffices for the bound in the supremum (in both $t$ and $U$) of
the second term in the right side of \eqref{normalformIBP}.

Similar to \eqref{nf2cut}, we estimate
\begin{multline*}
  \int\int_{C^U_\tau} \frac{1}{U} |u_{k-1}| |Z^{\le 60} u_{k-1}|
  |(\partial_t+\partial_r)(rZ^{\le 60} u_k)|\,dx\,dt
\\\lesssim \Bigl(\frac{1}{\tau^{\frac{1}{2}}U^{\frac{1}{2}}}\| Z^{\le 4}
u_{k-1}\|_{L^2L^2(\tC^U_\tau)} + \frac{1}{\tau^{\frac{1}{2}}U^{\frac{1}{2}}}\|(\partial_t+\partial_r)(r Z^{\le 3}
u_{k-1})\|_{L^2L^2(\tC^U_\tau)}\Bigr)  \frac{1}{\tau^{\frac{1}{2}}U^{\frac{1}{2}}} \|Z^{\le 60}
u_{k-1}\|_{L^2L^2(C^U_\tau)} \\\times \frac{1}{U^{\frac{1}{2}}\tau^{\frac{1}{2}}}\|(\partial_t+\partial_r)(r Z^{\le 60}
u_k)\|_{L^2L^2(C^U_\tau)}.
\end{multline*}
Since
\[\sigma_U'(t-r)\lesssim \frac{1}{\tau}\,\text{ on } C^R_\tau,\quad
  \sigma'_U(t-r)\lesssim \frac{1}{U}\,\text{ on } C^U_\tau,\]
we may combine this with \eqref{nf2crt} to see that
\begin{multline}
  \label{nf3}
  \int_0^T\int \sigma'_U(t-r) |u_{k-1}| |Z^{\le 60} u_{k-1}|
  |(\partial_t+\partial_r)(rZ^{\le 60} u_k)|\,dx\,dt
\\\lesssim (IV_{k-1}+III_{k-1})IV_{k-1} (III_k+IV_k)
+ (\log \la T\ra)^2 \Bigl(\log \la T\ra VIII_{k-1} + VII_{k-1}\Bigr)
VIII_{k-1} VII_k,
\end{multline}
which provides the appropriate bound for the supremums of the third
term in \eqref{normalformIBP}.

As \eqref{IIIb2} suffices to bound the fourth term in the right side of
\eqref{normalformIBP}, it remains to consider
\begin{multline}\label{nf4.1}
  \int_0^t\int e^{-\sigma_U(t-r)} a^{I,\alpha}_{JK}\omega_\alpha
  u_{k-1}^J Z^{\le 60} u_{k-1}^K (\partial_t^2-\partial_r^2)(rZ^{\le
    60} u_k^I)\,dx\,dt
\\=
\int_0^t\int e^{-\sigma_U(t-r)} a^{I,\alpha}_{JK}\omega_\alpha
  u_{k-1}^J Z^{\le 60} u_{k-1}^K \ang\cdot\ang (rZ^{\le
    60} u_k^I)\,dx\,dt
\\+\int_0^t\int r e^{-\sigma_U(t-r)} a^{I,\alpha}_{JK}\omega_\alpha
  u_{k-1}^J Z^{\le 60} u_{k-1}^K
  h^{I,\beta\gamma}_{\tilde{J}} \partial_\beta\partial_\gamma Z^{\le
    60} u_k^{\tilde{J}}\,dx\,dt
\\+\int_0^t\int r e^{-\sigma_U(t-r)} a^{I,\alpha}_{JK}\omega_\alpha
  u_{k-1}^J Z^{\le 60} u_{k-1}^K\Box_h Z^{\le
    60} u_k^I\,dx\,dt.
\end{multline}

For the first term in the right, we integrate by parts and use \eqref{angBound} to see that
\begin{multline*}
\sup_U\sup_{t\in [0,T]} \Bigl| \int_0^t\int e^{-\sigma_U(t-r)} a^{I,\alpha}_{JK}\omega_\alpha
  u_{k-1}^J Z^{\le 60} u_{k-1}^K \ang\cdot\ang (rZ^{\le
    60} u_k^I)\,dx\,dt\Bigr|
\\\lesssim
\int_0^T\int |Z^{\le 1} u_{k-1}| |Z^{\le 60} u_{k-1}| |\ang Z^{\le 60}
u_k|\,dx\,dt
+ \int_0^T\int r |u_{k-1}| |\ang Z^{\le 60} u_{k-1}| |\ang Z^{\le 60}
u_k|\,dx\,dt.
\end{multline*} 
 The preceding bound \eqref{IIIb1} shows that the latter term is
 controlled by \eqref{III_goal2}.  And \eqref{crt} (and \eqref{ctt}) and \eqref{cut},
 respectively, give
 \begin{multline*}
   \int\int_{C^R_\tau} |Z^{\le 1} u_{k-1}| |Z^{\le 60} u_{k-1}| |\ang Z^{\le 60}
u_k|\,dx\,dt 
\\\lesssim \Bigl(\|r^{-1} Z^{\le 5} u_{k-1}\|_{L^2L^2(\tC^R_\tau)} +
\|\good Z^{\le 4} u_{k-1}\|_{L^2L^2(\tC^R_\tau)}\Bigr) \|r^{-1} Z^{\le
  60} u_{k-1}\|_{L^2L^2(C^R_\tau)} \|\good Z^{\le 60} u_k\|_{L^2L^2(C^R_\tau)}
 \end{multline*}
and
\begin{multline*}
   \int\int_{C^U_\tau} |Z^{\le 1} u_{k-1}| |Z^{\le 60} u_{k-1}| |\ang Z^{\le 60}
u_k|\,dx\,dt
\\\lesssim \Bigl(\frac{1}{U^{\frac{1}{2}}\tau^{\frac{1}{2}}} \|Z^{\le 5}
  u_{k-1}\|_{L^2L^2(\tC^U_\tau)} +
  \frac{1}{U^{\frac{1}{2}}\tau^{\frac{1}{2}}}
  \|(\partial_t+\partial_r)(rZ^{\le 4}
  u_{k-1})\|_{L^2L^2(\tC^U_\tau)}\Bigr) \\\times \|r^{-1} Z^{\le 60}
  u_{k-1}\|_{L^2L^2(C^U_\tau)} \|\good Z^{\le 60} u_k\|_{L^2L^2(C^U_\tau)}.
 \end{multline*}
Hence,
  \[  \int_0^T\int |Z^{\le 1} u_{k-1}| |Z^{\le 60} u_{k-1}| |\ang Z^{\le 60}
u_k|\,dx\,dt
\lesssim (IV_{k-1}+III_{k-1}+\log\la T\ra VIII_{k-1}+VII_{k-1}) IV_{k-1} III_k,\]
which shows that 
\begin{equation}
  \label{nf4.2}
  \sup_U\sup_{t\in [0,T]} \Bigl| \int_0^t\int e^{-\sigma_U(t-r)} a^{I,\alpha}_{JK}\omega_\alpha
  u_{k-1}^J Z^{\le 60} u_{k-1}^K \ang\cdot\ang (rZ^{\le
    60} u_k^I)\,dx\,dt\Bigr| \lesssim \log\la T\ra M_{k-1}^2 M_k.
\end{equation}

Another integration by parts gives
\begin{multline}\label{nf4.2.b}
  \sup_U \sup_{t\in [0,T]}\Bigl|\int_0^t\int r e^{-\sigma_U(t-r)} a^{I,\alpha}_{JK}\omega_\alpha
  u_{k-1}^J Z^{\le 60} u_{k-1}^K
  h^{I,\beta\gamma}_{\tilde{J}} \partial_\beta\partial_\gamma Z^{\le
    60} u_k^{\tilde{J}}\,dx\,dt\Bigr|
\\\lesssim 
\sup_{t\in [0,T]} \int r |\partial^{\le 1}u_{k-1}|^2 |Z^{\le 60}
u_{k-1}| |\partial Z^{\le 60} u_k|\,dx
+
\int_0^T\int |\partial^{\le 1} u_{k-1}|^2 |Z^{\le 60}
u_{k-1}| |\partial Z^{\le 60} u_k|\,dx\,dt
\\+\int_0^T\int \frac{r}{\la t-r\ra} |\partial^{\le 1} u_{k-1}|^2
|Z^{\le 60} u_{k-1}| |\partial Z^{\le 60} u_k|\,dx\,dt
\\+\int_0^T\int r |\partial^{\le 1} u_{k-1}| |\partial \partial^{\le 1} u_{k-1}| |Z^{\le
  60} u_{k-1}| |\partial Z^{\le 60} u_k|\,dx\,dt
\\+ \int_0^T\int r |\partial^{\le 1} u_{k-1}|^2 |\partial Z^{\le 60}
u_{k-1}| |\partial Z^{\le 60} u_k|\,dx\,dt.
\end{multline}

If we argue precisely as in \eqref{nf1}, we see that
\begin{multline*}
  \sup_{t\in [0,T]} \int r |\partial^{\le 1}u_{k-1}|^2 |Z^{\le 60}
u_{k-1}| |\partial Z^{\le 60} u_k|\,dx \\\lesssim
\Bigl(IV_{k-1}+III_{k-1}+\log\la T\ra VIII_{k-1}+VII_{k-1}\Bigr)
VI_{k-1} \|\la r\ra^{\frac{1}{2}} |\partial^{\le 1} u_{k-1}| |\partial
Z^{\le 60}u_k|\|_{L^\infty L^2}.
\end{multline*}
Subsequently applying \eqref{weighted_Sobolev} gives that
\[\|\la r\ra^{\frac{1}{2}} |\partial^{\le 1} u_{k-1}| |\partial
Z^{\le 60}u_k|\|_{L^\infty L^2} \lesssim \|r^{-\frac{1}{2}} Z^{\le 3}
u_{k-1}\|_{L^\infty L^2} \|\partial Z^{\le 60} u_k\|_{L^\infty L^2}.\]
And hence,
\begin{multline}
  \label{nf4.2.a}
   \sup_{t\in [0,T]} \int r |\partial^{\le 1}u_{k-1}|^2 |Z^{\le 60}
u_{k-1}| |\partial Z^{\le 60} u_k|\,dx \\\lesssim
\Bigl(IV_{k-1}+III_{k-1}+\log\la T\ra VIII_{k-1}+VII_{k-1}\Bigr)
VI_{k-1} VI_{k-1} I_k.
\end{multline}

For the remaining terms in \eqref{nf4.2.b},
we continue to apply \eqref{crt}, \eqref{cut}, and \eqref{ctt} repeatedly.  These
give
\begin{multline*}
  \int\int_{C^R_\tau} |\partial^{\le 1} u_{k-1}|^2 |Z^{\le 60}
u_{k-1}| |\partial Z^{\le 60} u_k|\,dx\,dt
\\\lesssim \Bigl(\|r^{-1} Z^{\le 5} u_{k-1}\|_{L^2L^2(\tC^R_\tau)} +
\|\good Z^{\le 4} u_{k-1}\|_{L^2L^2(\tC^R_\tau)}\Bigr)^2 \|r^{-1}
Z^{\le 60} u_{k-1}\|_{L^2L^2(C^R_\tau)}
\frac{1}{R^{\frac{1}{2}}}\|r^{-\frac{1}{2}} \partial Z^{\le 60}
u_k\|_{L^2L^2(C^R_\tau)}
\end{multline*}
and
\begin{multline*}
  \int\int_{C^U_\tau} |\partial^{\le 1} u_{k-1}|^2 |Z^{\le 60}
u_{k-1}| |\partial Z^{\le 60} u_k|\,dx\,dt
\\\lesssim \Bigl(\frac{1}{U^{\frac{1}{2}}\tau^{\frac{1}{2}}} \|Z^{\le 5}
u_{k-1}\|_{L^2L^2(\tC^U_\tau)} +
\frac{1}{U^{\frac{1}{2}}\tau^{\frac{1}{2}}} \|(\partial_t+\partial_r)(r Z^{\le 4} u_{k-1})\|_{L^2L^2(\tC^U_\tau)}\Bigr)^2 \\\times\|r^{-1}
Z^{\le 60} u_{k-1}\|_{L^2L^2(C^U_\tau)} \frac{1}{\tau^{\frac{1}{2}}}\|r^{-\frac{1}{2}} \partial Z^{\le 60} u_k\|_{L^2L^2(C^U_\tau)}.
\end{multline*}
Similarly,
\begin{multline*}
  \int\int_{C^U_\tau} \frac{r}{\la t-r\ra} |\partial^{\le 1} u_{k-1}|^2 |Z^{\le 60}
u_{k-1}| |\partial Z^{\le 60} u_k|\,dx\,dt
\\\lesssim \Bigl(\frac{1}{U^{\frac{1}{2}}\tau^{\frac{1}{2}}} \|Z^{\le 5}
u_{k-1}\|_{L^2L^2(\tC^U_\tau)} +
\frac{1}{U^{\frac{1}{2}}\tau^{\frac{1}{2}}} \|(\partial_t+\partial_r) (r Z^{\le 4}
u_{k-1})\|_{L^2L^2(\tC^U_\tau)}\Bigr)^2 \\\times U^{-1}\|r^{-\frac{1}{2}}
Z^{\le 60} u_{k-1}\|_{L^2L^2(C^U_\tau)} \|r^{-\frac{1}{2}} \partial Z^{\le 60} u_k\|_{L^2L^2(C^U_\tau)}.
\end{multline*}
These combine to give
\begin{multline}\label{nf4.2.1}
  \int_0^T\int |\partial^{\le 1} u_{k-1}|^2 |Z^{\le 60}
u_{k-1}| |\partial Z^{\le 60} u_k|\,dx\,dt + \int\int_{C^U_\tau} \frac{r}{\la t-r\ra} |\partial^{\le 1} u_{k-1}|^2 |Z^{\le 60}
u_{k-1}| |\partial Z^{\le 60} u_k|\,dx\,dt
\\\lesssim (IV_{k-1}+III_{k-1})^2IV_{k-1}II_k + (\log\la T\ra
VIII_{k-1}+VII_{k-1})^2 \Bigl(IV_{k-1} II_k + \log\la T\ra VIII_{k-1} II_k\Bigr) .
\end{multline}
By related arguments,
\begin{multline*}
  \int\int_{C^R_\tau} r |\partial^{\le 1} u_{k-1}| |\partial \partial^{\le 1} u_{k-1}| |Z^{\le
  60} u_{k-1}| |\partial Z^{\le 60} u_k|\,dx\,dt
\\\lesssim \Bigl(\|r^{-1} Z^{\le 5} u_{k-1}\|_{L^2L^2(\tC^R_\tau)} +
\|\good Z^{\le 4} u_{k-1}\|_{L^2L^2(\tC^R_\tau)}\Bigr) \|r^{-1} Z^{\le 60} u_{k-1}\|_{L^2L^2(C^R_\tau)}\\\times
\Bigl(\|\partial Z^{\le 5} u_{k-1}\|_{L^2L^2(\tC^R_\tau)} + R
\|\good \partial Z^{\le 4} u_{k-1}\|_{L^2L^2(\tC^R_\tau)}\Bigr)
R^{-\frac{1}{2}} \|\la r\ra^{-\frac{1}{2}}\partial Z^{\le 60} u_k\|_{L^2L^2(C^R_\tau)}
\end{multline*}
and
\begin{multline*}
  \int\int_{C^U_\tau} r |\partial^{\le 1} u_{k-1}| |\partial \partial^{\le 1} u_{k-1}| |Z^{\le
  60} u_{k-1}| |\partial Z^{\le 60} u_k|\,dx\,dt
\\\lesssim \frac{1}{U^{\frac{1}{2}}}\Bigl(\frac{1}{U^{\frac{1}{2}}\tau^{\frac{1}{2}}} \|Z^{\le
  5} u_{k-1}\|_{L^2L^2(\tC^U_\tau)} +
\frac{1}{U^{\frac{1}{2}}\tau^{\frac{1}{2}}}
  \|(\partial_t+\partial_r)(rZ^{\le 4}
  u_{k-1})\|_{L^2L^2(\tC^U_\tau)}\Bigr)
\\\times
\Bigl(\frac{U^{\frac{1}{2}}}{\tau^{\frac{1}{2}}}\|\partial Z^{\le 5}
u_{k-1}\|_{L^2L^2(\tC^U_\tau)} + U^{\frac{1}{2}}\tau^{\frac{1}{2}} \|\good \partial Z^{\le 4}
u_{k-1}\|_{L^2L^2(\tC^U_\tau)}\Bigr) \\\times \frac{1}{U^{\frac{1}{2}}\tau^{\frac{1}{2}}}\|Z^{\le 60}
u_{k-1}\|_{L^2L^2(C^U_\tau)} \frac{1}{\tau^{\frac{1}{2}}}\|\partial Z^{\le 60} u_k\|_{L^2L^2(C^U_\tau)},
\end{multline*}
which gives
\begin{multline}\label{nf4.2.2}
  \int_0^T\int r |\partial^{\le 1} u_{k-1}| |\partial \partial^{\le 1} u_{k-1}| |Z^{\le
  60} u_{k-1}| |\partial Z^{\le 60} u_k|\,dx\,dt
\\\lesssim (IV_{k-1}+III_{k-1}) (IX_{k-1}+X_{k-1}) IV_{k-1} II_k
\\+ (\log\la T\ra )^2(\log\la T\ra VIII_{k-1} + VII_{k-1}) 
(XI_{k-1}+XII_{k-1}) VIII_{k-1} II_k.
\end{multline}
For the last term in \eqref{nf4.2.b}, we get
\begin{multline*}
  \int\int_{C^R_\tau} r |\partial^{\le 1} u_{k-1}|^2 |\partial Z^{\le 60}
u_{k-1}| |\partial Z^{\le 60} u_k|\,dx\,dt
\lesssim \Bigl(\|r^{-1} Z^{\le 5} u_{k-1}\|_{L^2L^2(\tC^R_\tau)} +
\|\good Z^{\le 4} u_{k-1}\|_{L^2L^2(\tC^R_\tau)}\Bigr)^2
\\\times\|r^{-\frac{1}{2}} \partial
Z^{\le 60} u_{k-1}\|_{L^2L^2(C^R_\tau)} \|r^{-\frac{1}{2}} \partial Z^{\le 60} u_k\|_{L^2L^2(C^R_\tau)}
\end{multline*}
and
\begin{multline*}
  \int\int_{C^U_\tau} r |\partial^{\le 1} u_{k-1}|^2 |\partial Z^{\le 60}
u_{k-1}| |\partial Z^{\le 60} u_k|\,dx\,dt
\\\lesssim \Bigl(\frac{1}{U^{\frac{1}{2}}\tau^{\frac{1}{2}}} \|Z^{\le 5}
u_{k-1}\|_{L^2L^2(\tC^U_\tau)} +
\frac{1}{U^{\frac{1}{2}}\tau^{\frac{1}{2}}} \|(\partial_t+\partial_r)(r Z^{\le 4} u_{k-1})\|_{L^2L^2(\tC^U_\tau)}\Bigr)^2 \\\times\|r^{-\frac{1}{2}}\partial
Z^{\le 60} u_{k-1}\|_{L^2L^2(C^U_\tau)} \|r^{-\frac{1}{2}} \partial Z^{\le 60} u_k\|_{L^2L^2(C^U_\tau)},
\end{multline*}
yielding
\begin{multline}\label{nf4.2.3}
   \int_0^T\int r |\partial^{\le 1} u_{k-1}|^2 |\partial Z^{\le 60}
u_{k-1}| |\partial Z^{\le 60} u_k|\,dx\,dt
\\\lesssim  (IV_{k-1}+III_{k-1})^2 II_{k-1} II_k + \log\la T\ra  (\log\la T\ra
VIII_{k-1} + VII_{k-1})^2 II_{k-1} II_k.
\end{multline}
The combination of \eqref{nf4.2.a}, \eqref{nf4.2.1}, \eqref{nf4.2.2}, and
\eqref{nf4.2.3} then establish that
\begin{equation}\label{nf4.2.4}
  \sup_U \sup_{t\in [0,T]}\Bigl|\int_0^t\int r e^{-\sigma_U(t-r)} a^{I,\alpha}_{JK}\omega_\alpha
  u_{k-1}^J Z^{\le 60} u_{k-1}^K
  h^{I,\alpha\beta}_{\tilde{J}} \partial_\alpha\partial_\beta Z^{\le
    60} u_k^{\tilde{J}}\,dx\,dt\Bigr|
\lesssim (\log\la T\ra)^3 M_{k-1}^3 M_k.
\end{equation}

We now turn our attention to
\[\sup_U \sup_{t\in [0,T]} \Bigl|\int_0^t \int r e^{-\sigma_U(t-r)}
  a^{I,\alpha}_{JK} \omega_\alpha u^J_{k-1} Z^{\le 60} u^K_{k-1}
  \Box_h Z^{\le 60} u_k^I\,dx\,dt\Bigr|.\]
We shall first show
\begin{multline}\label{nf5.goal} \sup_U \sup_{t\in [0,T]} \Bigl|\int_0^t \int r e^{-\sigma_U(t-r)}
  a^{I,\alpha}_{JK} \omega_\alpha u^J_{k-1} Z^{\le 60} u^K_{k-1}
  \Bigl(\Box_h Z^{\le 60} u_k^I -a^{I,\beta}_{\tJ\tK}\omega_\beta
  u^\tJ_{k-1}(\partial_t-\partial_r)Z^{\le 60} u^\tK_{k-1}\Bigr) \,dx\,dt\Bigr|
\\\lesssim (\log \la T\ra)^5 M_{k-1}^4 + (\log \la T\ra)^5 M_{k-1}^3M_k.
\end{multline}
By \eqref{Box.minus.udu}, we have
\begin{multline}\label{nf5.1} \sup_U \sup_{t\in [0,T]} \Bigl|\int_0^t \int r e^{-\sigma_U(t-r)}
  a^{I,\alpha}_{JK} \omega_\alpha u^J_{k-1} Z^{\le 60} u^K_{k-1}
  \Bigl(\Box_h Z^{\le 60} u_k^I -a^{I,\beta}_{\tJ\tK}\omega_\beta
  u^\tJ_{k-1}(\partial_t-\partial_r)Z^{\le 60} u^\tK_{k-1}\Bigr) \,dx\,dt\Bigr|
\\\lesssim \int_0^T\int r |u_{k-1}| |Z^{\le 60} u_{k-1}| |Z^{\le 30}
u_{k-1}| |\good Z^{\le 60} u_{k-1}|\,dx\,dt
\\+ \int_0^T\int r |u_{k-1}| |Z^{\le 60} u_{k-1}| |\partial Z^{\le 30}
u_{k-1}| |\partial^{\le 1} Z^{\le 60} u_{k-1}|\,dx\,dt
\\+\int_0^T\int r |u_{k-1}| |Z^{\le 60} u_{k-1}| |\partial Z^{\le 31}
u_{k}| |\partial^{\le 1} Z^{\le 60} u_{k-1}|\,dx\,dt
\\\int_0^T\int r |u_{k-1}| |Z^{\le 60} u_{k-1}| |\partial^{\le 1} Z^{\le 30}
u_{k-1}| |\partial^{2} Z^{\le 59} u_{k}|\,dx\,dt.
\end{multline}
We note that \eqref{crt} and \eqref{ctt} give
\begin{equation}
  \label{nf5.a}
  \| |u_{k-1}| |Z^{\le 60}u_{k-1}|\|_{L^2L^2(C^R_\tau)} \lesssim
  \Bigl(\|r^{-1} Z^{\le 4}u_{k-1}\|_{L^2L^2(\tC^R_\tau)} + \|\good
  Z^{\le 3} u_{k-1}\|_{L^2L^2(\tC^R_\tau)}\Bigr) \|r^{-1} Z^{\le 60}
  u_{k-1}\|_{L^2L^2(C^R_\tau)}
\end{equation}
and \eqref{cut} gives
\begin{multline}
  \label{nf5.b}
  \frac{\tau^{\frac{1}{2}}}{U^{\frac{1}{2}}} \||u_{k-1}| |Z^{\le 60}
  u_{k-1}|\|_{L^2L^2(C^U_\tau)}
\\\lesssim \frac{1}{\tau^{\frac{1}{2}}U^{\frac{1}{2}}}\Bigl(\|Z^{\le
  4} u_{k-1}\|_{L^2L^2(\tC^U_\tau)} +
\|(\partial_t+\partial_r)(rZ^{\le 3}
u_{k-1})\|_{L^2L^2(\tC^U_\tau)}\Bigr) \frac{1}{U^{\frac{1}{2}}
  \tau^{\frac{1}{2}}} \|Z^{\le 60} u_{k-1}\|_{L^2L^2(C^U_\tau)}.
\end{multline}
Arguing as in \eqref{IIIb1}, \eqref{IIIb2}, \eqref{IIIb3}, and
\eqref{IIIb4} where 
\[\|\good Z^{\le 60} u_k\|_{L^2L^2(C^R_\tau)}, \quad
  \frac{1}{U^{\frac{1}{2}}\tau^{\frac{1}{2}}}
  \|(\partial_t+\partial_r)(rZ^{\le 60} u_k)\|_{L^2L^2(C^U_\tau)} \]
are replaced by \eqref{nf5.a} and \eqref{nf5.b}, respectively, we
immediately get \eqref{nf5.goal}.

To complete the proof of \eqref{IIIBox}, we now consider
\begin{multline*}
  \int_0^t \int r e^{-\sigma_U(t-r)}
  a^{I,\alpha}_{JK} \omega_\alpha u^J_{k-1} Z^{\le 60} u^K_{k-1}
  \Bigl(a^{I,\beta}_{\tJ\tK}\omega_\beta
  u^\tJ_{k-1}(\partial_t-\partial_r)Z^{\le 60} u^\tK_{k-1}\Bigr) \,dx\,dt
  \\= \frac{1}{2}
   \int_0^t \int r e^{-\sigma_U(t-r)}
 u^J_{k-1}
  u^\tJ_{k-1}(\partial_t-\partial_r)\Bigl[  a^{I,\alpha}_{JK}
  \omega_\alpha
  a^{I,\beta}_{\tJ\tK}\omega_\beta
   Z^{\le 60} u^K_{k-1}
  Z^{\le 60} u^\tK_{k-1}\Bigr] \,dx\,dt.
\end{multline*}
Integration by parts shows that
\begin{multline}\label{nf6.1}
\sup_U \sup_{t\in [0,T]}\Bigl|\int_0^t \int r e^{-\sigma_U(t-r)}
  a^{I,\alpha}_{JK} \omega_\alpha u^J_{k-1} Z^{\le 60} u^K_{k-1}
  \Bigl(a^{I,\beta}_{\tJ\tK}\omega_\beta
  u^\tJ_{k-1}(\partial_t-\partial_r)Z^{\le 60} u^\tK_{k-1}\Bigr)
  \,dx\,dt\Bigr|
  \\\lesssim
  \sup_{t\in [0,T]} \int r |u_{k-1}|^2 |Z^{\le 60} u_{k-1}|^2\,dx
  + \int_0^T\int \frac{r}{\la t-r\ra} |u_{k-1}|^2 |Z^{\le 60}
u_{k-1}|^2\,dx\,dt
\\+ \int_0^T\int |u_{k-1}|^2 |Z^{\le 60}
u_{k-1}|^2\,dx\,dt
+ \int_0^T\int r |u_{k-1}| |\partial u_{k-1}| |Z^{\le 60} u_{k-1}|^2\,dx\,dt.
\end{multline}

For the first term, we first consider a fixed $t$ and set $\tau\approx
t$.  Then by \eqref{crt} (and \eqref{ctt}),
\[  \int_{A_R} r |u_{k-1}|^2 |Z^{\le 60} u_{k-1}|^2\,dx
  \lesssim \Bigl(\|r^{-1}Z^{\le 4}u_{k-1}\|_{L^2L^2(\tC^R_\tau)} +
  \|\good Z^{\le 3} u_{k-1}\|_{L^2L^2(\tC^R_\tau)}\Bigr)^2
  \|r^{-\frac{1}{2}} Z^{\le 60} u_{k-1}(t,\cd)\|_{L^2}^2,
\]
and by \eqref{cut}
\begin{multline*}
  \int_{\la t-r\ra\approx U}  r |u_{k-1}|^2 |Z^{\le 60} u_{k-1}|^2\,dx
  \lesssim \Bigl(\frac{1}{U^{\frac{1}{2}}\tau^{\frac{1}{2}}} \|Z^{\le
    4} u_{k-1}\|_{L^2L^2(\tC^U_\tau)} +
  \frac{1}{U^{\frac{1}{2}}\tau^{\frac{1}{2}}}
  \|(\partial_t+\partial_r)rZ^{\le 3}
  u_{k-1}\|_{L^2L^2(\tC^U_\tau)}\Bigr)^2\\\times \|r^{-\frac{1}{2}} Z^{\le 60}
  u_{k-1}(t,\cd)\|^2_{L^2(\{\la t-r\ra\approx U\})}. 
\end{multline*}
Upon summing over $R\le \tau/2$, $U\le \tau/4$ and then taking a supremum in $t$,
we obtain
\begin{equation}
  \label{nf6.2}
  \sup_{t\in [0,T]} \int r |u_{k-1}|^2 |Z^{\le 60} u_{k-1}|^2\,dx
  \lesssim (IV_{k-1}+III_{k-1})^2VI_{k-1}^2 + (\log\la T\ra
  VIII_{k-1}+VII_{k-1})^2 VI_{k-1}^2.
\end{equation}

For the second term in the right of \eqref{nf6.1}, subsequently applying
\eqref{crt} (and \eqref{ctt}) and \eqref{cut} when $R \le \tau/2$, $U\le \tau/4$ gives
\begin{multline}\label{nf6.3.a}
  \int\int_{C^R_\tau} \frac{r}{\la t-r\ra} |u_{k-1}|^2|Z^{\le
    60}u_{k-1}|^2\,dx\,dt
  \\\lesssim \Bigl(\|r^{-1} Z^{\le 4} u_{k-1}\|_{L^2L^2(\tC^R_\tau)} +
  \|\good Z^{\le 3} u_{k-1}\|_{L^2L^2(\tC^R_\tau)}\Bigr)^2 \|r^{-1}
  Z^{\le 60} u_{k-1}\|^2_{L^2L^2(C^R_\tau)}
\end{multline}
and
\begin{multline*}
  \int\int_{C^U_\tau} \frac{r}{\la t-r\ra} |u_{k-1}|^2|Z^{\le 60}u_{k-1}|^2\,dx\,dt
  \\\lesssim
\Bigl(\frac{1}{U^{\frac{1}{2}}\tau^{\frac{1}{2}}} \|Z^{\le 4}
u_{k-1}\|_{L^2L^2(\tC^U_\tau)} +
\frac{1}{U^{\frac{1}{2}}\tau^{\frac{1}{2}}}
\|(\partial_t+\partial_r)(rZ^{\le 3}
u_{k-1})\|_{L^2L^2(\tC^U_\tau)}\Bigr)^2\\\times
\Bigl(\frac{1}{U^{\frac{1}{2}}\tau^{\frac{1}{2}}} \|Z^{\le 60}
u_{k-1}\|_{L^2L^2(C^U_\tau)}\Bigr)^2.  
\end{multline*}
Upon summing, we obtain
\begin{multline}\label{nf6.3}
  \int_0^T\int \frac{r}{\la t-r\ra} |u_{k-1}|^2|Z^{\le
    60}u_{k-1}|^2\,dx\,dt
  \lesssim (IV_{k-1}+III_{k-1})^2 IV_{k-1}^2 \\+ (\log \la T\ra)^3
  \Bigl(\log \la T\ra VIII_{k-1} + VII_{k-1}\Bigr)^2 VIII_{k-1}^2.
\end{multline}
Moreover,
\begin{multline*}
  \int\int_{C^U_\tau} |u_{k-1}|^2|Z^{\le 60}u_{k-1}|^2\,dx\,dt
  \\\lesssim
\Bigl(\frac{1}{U^{\frac{1}{2}}\tau^{\frac{1}{2}}} \|Z^{\le 4}
u_{k-1}\|_{L^2L^2(\tC^U_\tau)} +
\frac{1}{U^{\frac{1}{2}}\tau^{\frac{1}{2}}}
\|(\partial_t+\partial_r)(rZ^{\le 3}
u_{k-1})\|_{L^2L^2(\tC^U_\tau)}\Bigr)^2
\|r^{-1}Z^{\le 60}
u_{k-1}\|_{L^2L^2(C^U_\tau)}^2.  
\end{multline*}
Hence when combined with \eqref{nf6.3.a}, this gives
\begin{equation}\label{nf6.4}
  \int_0^T\int  |u_{k-1}|^2|Z^{\le
    60}u_{k-1}|^2\,dx\,dt
  \lesssim (IV_{k-1}+III_{k-1})^2 IV_{k-1}^2 + 
  \Bigl(\log \la T\ra VIII_{k-1} + VII_{k-1}\Bigr)^2 IV_{k-1}^2.
\end{equation}

We may use the estimate on the second term in the right side of \eqref{nf5.1} to bound the last term in \eqref{nf6.1}.
Combining this with \eqref{nf6.2}, \eqref{nf6.3}, and \eqref{nf6.4}
then yields
\[
\sup_U \sup_{t\in [0,T]}\Bigl|\int_0^t \int r e^{-\sigma_U(t-r)}
  a^{I,\alpha}_{JK} \omega_\alpha u^J_{k-1} Z^{\le 60} u^K_{k-1}
  \Bigl(a^{I,\beta}_{\tJ\tK}\omega_\beta
  u^\tJ_{k-1}(\partial_t-\partial_r)Z^{\le 60} u^\tK_{k-1}\Bigr)
  \,dx\,dt\Bigr|
\lesssim (\log \la T\ra)^5 M_{k-1}^4,\]
which completes the proof of \eqref{IIIBox}.
\end{proof}

To finish the proof of \eqref{III_goal}, we consider the
remainder of the terms in the right side of \eqref{newLE} and show
that they are each controlled by \eqref{III_goal2}.

By \eqref{crt} and \eqref{ctt},
\begin{multline*}
  \int\int_{C^R_\tau} \Bigl(|\partial Z^{\le 1} u_{k-1}| + \frac{|Z^{\le 1}
    u_{k-1}|}{r}\Bigr) |\partial Z^{\le 60} u_k|
  |(\partial_t+\partial_r)(rZ^{\le 60} u_k)|\,dx\,dt  \\ \lesssim
  \Bigl( \|r^{-1} Z^{\le 5} u_{k-1}\|_{L^2L^2(\tilde{C}^R_\tau)} + \|\partial Z^{\le 5}
  u_{k-1}\|_{L^2L^2(\tilde{C}^R_\tau)} +R\|\good \partial Z^{\le
    4} u_{k-1}\|_{L^2L^2(\tilde{C}^R_\tau)}\Bigr) \\\times\| r^{-1}\partial Z^{\le
    60} u_k\|_{L^2L^2(C^R_\tau)} \Bigl(\|\good Z^{\le 60}
  u_k\|_{L^2L^2(C^R_\tau)} + \|r^{-1} Z^{\le 60} u_k\|_{L^2L^2(C^R_\tau)}\Bigr).
\end{multline*}
And by \eqref{cut},
\begin{multline*}
  \int\int_{C^U_\tau} \Bigl(|\partial Z^{\le 1} u_{k-1}| + \frac{|Z^{\le 1}
    u_{k-1}|}{r}\Bigr) |\partial Z^{\le 60} u_k|
  |(\partial_t+\partial_r)(rZ^{\le 60} u_k)|\,dx\,dt  \\\lesssim
  \Bigl(\| r^{-1} Z^{\le 5}
  u_{k-1}\|_{L^2L^2(\tilde{C}^U_\tau)}  + \frac{U^{\frac{1}{2}}}{\tau^{\frac{1}{2}}}\|\partial Z^{\le 5}
  u_{k-1}\|_{L^2L^2(\tilde{C}^U_\tau)} +U^{\frac{1}{2}}\tau^{\frac{1}{2}}\|\good \partial Z^{\le
   4} u_{k-1}\|_{L^2L^2(\tilde{C}^U_\tau)}\Bigr) \\\times\frac{1}{U^{\frac{1}{2}}}\|\la r\ra^{-\frac{1}{2}}\partial Z^{\le
    60} u_k\|_{L^2L^2(C^U_\tau)} \frac{1}{U^{\frac{1}{2}}\tau^{\frac{1}{2}}}\|(\partial_t+\partial_r) (r Z^{\le 60}
  u_k)\|_{L^2L^2(C^U_\tau)}.
\end{multline*}
Upon summing over $R \le \tau/2$, $U \le \tau/4$ and $\tau \le T$, we get
\begin{multline}\label{term4}  \int_0^T \int \Bigl(|\partial Z^{\le 1} u_{k-1}| + \frac{|Z^{\le 1}
    u_{k-1}|}{r}\Bigr) |\partial Z^{\le 60} u_k|
  |(\partial_t+\partial_r)(rZ^{\le 60} u_k)|\,dx\,dt
  \\\lesssim (IV_{k-1}+IX_{k-1}+X_{k-1})II_{k} (III_k + IV_k)
  \\+  \Bigl(IV_{k-1} + \log\la T\ra XI_{k-1} + \log \la
  T\ra XII_{k-1}\Bigr)  II_{k} VII_k.
\end{multline}

For the fifth term in the right side of \eqref{newLE},
applying \eqref{weighted_Sobolev} and the Schwarz inequality, we have
\begin{equation}\label{term5}\begin{split}
  \int_0^T \int |\partial^{\le 1} u_{k-1}| |\partial Z^{\le 60} &u_k| \Bigl(|\ang Z^{\le 60}
  u_k| + \frac{|Z^{\le 60} u_k|}{r}\Bigr)\,dx\,dt
 \\ &\lesssim \|\la r\ra^{-1} Z^{\le 3} u_{k-1}\|_{L^2L^2}
  \|\partial Z^{\le 60} u_k\|_{L^\infty L^2}\Bigl(\|\ang Z^{\le
    60}u_k\|_{L^2L^2} + \|r^{-1} Z^{\le 60} u_k\|_{L^2L^2}\Bigr)\\
  &\lesssim   IV_{k-1} I_k (III_k+IV_k).
\end{split}
\end{equation}
And for the sixth term, using \eqref{weighted_Sobolev},
\begin{equation}\label{term6}\begin{split}
  \int_0^T\int \la t-r\ra^{-1}|&\partial^{\le 1} u_{k-1}| |\partial
  Z^{\le 60} u_k| |(\partial_t+\partial_r)(rZ^{\le 60}u_k)|\,dx\,dt
  \\&\lesssim \log\la T\ra \Bigl(\sup_U \|\la r\ra^{-\frac{1}{2}} \la t-r\ra^{-\frac{1}{2}}
  Z^{\le 3} u_{k-1}\|_{L^2L^2(X_U)}\Bigr) \|\partial Z^{\le
    60} u_k\|_{L^\infty L^2} \\&\qquad\qquad\qquad\qquad\qquad\qquad\times\sup_U\Bigl(\|\la r\ra^{-\frac{1}{2}} \la
  t-r\ra^{-\frac{1}{2}} (\partial_t+\partial_r)(r Z^{\le
  60}u_k)\|_{L^2L^2(X_U)}\Bigr)\\
  &\lesssim (\log \la T\ra)^2 VIII_{k-1}\cdot I_k\cdot VII_k. 
\end{split}
\end{equation}

For the seventh term in \eqref{newLE}, we apply \eqref{crt} (and \eqref{ctt}), which yields
\begin{multline*}
\int\int_{C^R_\tau} r \Bigl(|(\partial_t+\partial_r) \partial^{\le 1}
  u_{k-1}| + \frac{|\partial^{\le 1} u_{k-1}|}{r}\Bigr) |\partial
  Z^{\le 60} u_k|^2\,dx\,dt
  \\\lesssim  \Bigl(\|\good Z^{\le 5}
  u_{k-1}\|_{L^2L^2(\tilde{C}^R_\tau)}
  + R\|\good \partial Z^{\le 4} u_{k-1}\|_{L^2L^2(\tilde{C}^R_\tau)}
  + \|r^{-1} Z^{\le 5} u_{k-1}\|_{L^2L^2(\tilde{C}^R_\tau)}\Bigr)
\\\times  \|\la r\ra^{-\frac{1}{2}} \partial Z^{\le 60}
  u_k\|^2_{L^2L^2(C^R_\tau)}.
\end{multline*}
On the $C^U_\tau$ regions, in addition to \eqref{cut}, we will also
apply \eqref{cut2} to the term that already contains a good
derivative, which gives
  \begin{multline*}
  \int \int_{C^U_\tau} r \Bigl(|(\partial_t+\partial_r) \partial^{\le 1}
  u_{k-1}| + \frac{|\partial^{\le 1} u_{k-1}|}{r}\Bigr) |\partial
  Z^{\le 60} u_k|^2\,dx\,dt
\\  \lesssim \Bigl(\frac{1}{U^{\frac{1}{2}}\tau^{\frac{1}{2}}}
\|Z^{\le 5} u_{k-1}\|_{L^2L^2(\tC^U_\tau)} +
\frac{1}{U^{\frac{1}{2}}\tau^{\frac{1}{2}}} \|(\partial_t+\partial_r) (r Z^{\le
  4}u_{k-1})\|_{L^2L^2(\tC^U_\tau)} +
\frac{U^{\frac{1}{2}}}{\tau^{\frac{1}{2}}}
\|\partial (\partial_t+\partial_r)(rZ^{\le 4} u_{k-1})\|_{L^2L^2(\tC^U_\tau)} \Bigr)
\\\times \|\la r\ra^{-\frac{1}{2}} \partial Z^{\le 60}
  u_k\|^2_{L^2L^2(C^U_\tau)}.
\end{multline*}
Using \eqref{angCommutator}, note that
\[
  \frac{U^{\frac{1}{2}}}{\tau^{\frac{1}{2}}} \|\partial
  (\partial_t+\partial_r) (rZ^{\le 4} u_{k-1})\|_{L^2L^2(\tC^U_\tau)}
  \lesssim \frac{U^{\frac{1}{2}}}{\tau^{\frac{1}{2}}} \|\partial
  Z^{\le 4} u_{k-1}\|_{L^2L^2(\tC^U_\tau)} + U^{\frac{1}{2}}
  \tau^{\frac{1}{2}} \|\good \partial Z^{\le 4} u_{k-1}\|_{L^2L^2(\tC^U_\tau)}.
\]
Thus, upon summing and using \eqref{L2L2toLE}, we obtain
\begin{multline}
  \label{term7}
  \int_0^T \int r \Bigl(|(\partial_t+\partial_r) \partial^{\le 1}
  u_{k-1}| + \frac{|\partial^{\le 1} u_{k-1}|}{r}\Bigr) |\partial
  Z^{\le 60} u_k|^2\,dx\,dt
  \\\lesssim (III_{k-1} + X_{k-1} + IV_{k-1}) (\log \la
  T\ra)^{\frac{1}{2}} II_k^2
  \\+\Bigl[VII_{k-1} +(\log\la T\ra)  \Bigl(VIII_{k-1} +
  XI_{k-1}+XII_{k-1}\Bigr)\Bigr]\log \la T\ra II_k^2.
\end{multline}

The desired bound \eqref{III_goal} now follows from \eqref{newLE} and
the application of \eqref{smallness}, \eqref{III.tbdy},
\eqref{IIIBox}, \eqref{term4}, \eqref{term5}, \eqref{term6},
\eqref{term7}, and \eqref{I+II_goal}.

\medskip
\noindent$\mathbf{[IX_k]}$:
As $R<\tau/2$ on each $C^R_\tau$, \eqref{derivative_decay} allows us to see that
\begin{equation}\label{VII}IX_k^2
  \lesssim \sum_{\tau\le T}\sum_{R\le \tau/4}
  \Bigl(\frac{1}{R^2}\|Z^{\le 51} u_k\|^2_{L^2L^2(\tilde{C}^R_\tau)} + \|
  \good Z^{\le 50} u_k\|^2_{L^2L^2(\tilde{C}^R_\tau)}\Bigr)\lesssim III_k^2 +
  IV_k^2. 
\end{equation}
And, thus, the appropriate bound is a consequence of \eqref{III_goal}.

\medskip
\noindent$\mathbf{[X_k]}$: Using \eqref{angBound} and
\eqref{derivative_decay}, when $R<\tau/2$, we see that
\begin{equation}\label{VIII.1}\begin{split}R\| 
    \good \partial Z^{\le 40} u_k\|_{L^2L^2(\tilde{C}^R_\tau)}
&\lesssim \|\partial Z^{\le 41} u_k\|_{L^2L^2(\tC^R_\tau)}
+ R\|
  (\partial_t+\partial_r)(\partial_t-\partial_r) Z^{\le 40} u_k\|_{L^2L^2(\tC^R_\tau)}\\
&\lesssim \| \partial Z^{\le 41} u_k\|_{L^2L^2(\tC^R_\tau)} + R\|
  \Box Z^{\le 40} u_k\|_{L^2L^2(\tC^R_\tau)}.
\end{split}
\end{equation}
On $\tC^R_\tau$, \eqref{VIII.1} is akin to the bounds of \cite{KlSid} and \cite{Sideris-Thomases}.
We note that 
\begin{equation}
  \label{lowOrderBox}
  | \Box Z^{\le 40} u_k| \lesssim \Bigl(|\partial Z^{\le 20} u_{k-1}| +
  | \partial Z^{\le 21}  u_k|\Bigr) |Z^{\le 41} u_{k-1}|
+ |Z^{\le 21} u_{k-1}|\Bigl(|\partial Z^{\le 40}  u_{k-1}| + |\partial Z^{\le
  41}   u_k|\Bigr).\end{equation}
Applying \eqref{crt}, we obtain
\begin{multline*}
  R\|Z^{\le 40}
  \Box u_k\|_{L^2L^2(\tC^R_\tau)} \lesssim R^{-1} \|Z^{\le 41}
  u_{k-1}\|_{L^2L^2(\tC^R_\tau)} \Bigl(\| \partial Z^{\le 24}
  u_{k-1}\|_{L^2L^2(\ttC^R_\tau)}
  + R\|\good\partial Z^{\le 23} u_{k-1}\|_{L^2L^2(\ttC^R_\tau)}
  \Bigr)
  \\
 + R^{-1} \|Z^{\le 41}
  u_{k-1}\|_{L^2L^2(\tC^R_\tau)} \Bigl(\|\partial Z^{\le 25} 
  u_{k}\|_{L^2L^2(\ttC^R_\tau)}
  + R\|\good\partial Z^{\le 24}u_{k}\|_{L^2L^2(\ttC^R_\tau)}
  \Bigr)
  \\
  +\Bigl(R^{-1} \|Z^{\le 25} u_{k-1}\|_{L^2L^2(\ttC^R_\tau)} +
  \|\good Z^{\le 24} u_{k-1}\|_{L^2L^2(\ttC^R_\tau)}\Bigr)
  \Bigl(\|\partial Z^{\le 40} u_{k-1}\|_{L^2L^2(\tC^R_\tau)} + \|\partial Z^{\le 41}
  u_k\|_{L^2L^2(\tC^R_\tau)}\Bigr).
\end{multline*}
Combining this with \eqref{VIII.1}, we see that
\begin{multline}\label{VIII}X_k \lesssim IX_k + IV_{k-1}\Bigl(IX_{k-1} +
  X_{k-1}+\log\la T\ra (XI_{k-1}+XII_{k-1})\Bigr) \\+
  IV_{k-1}\Bigl(IX_k + X_k+\log\la T\ra(XI_k+XII_k)\Bigr)
  + \Bigl(IV_{k-1} + III_{k-1}\Bigr)\Bigl(IX_{k-1} + IX_k\Bigr).
\end{multline}
We note that the occurrences of terms $XI$ and $XII$ are due to the
enlargement of $\tC^R_\tau$ when $R=\tau/4$ and $R= \tau/2$.  The tails here can be
bounded using the $C^U_\tau$ terms when $U=\tau/4$.  By \eqref{III_goal}, these terms are also bounded by the right side
of \eqref{MkGoal}.

\medskip
\noindent$\mathbf{[XI_k]}$:
For $U=1$, we have an immediate bound
 \[\sum_{\tau} \frac{1}{\tau (\log\la \tau\ra)^2} \|\partial Z^{\le 50} 
 u_k\|^2_{L^2L^2(\tC^{U=1}_\tau)} \lesssim \sum_{\tau} \frac{1}{(\log\la
   \tau\ra)^2} \|\partial Z^{\le 50}
 u_k\|^2_{L^\infty_tL^2_x} \le I^2_k.\]
For $U>1$, we can refer to \eqref{angBound} and \eqref{derivative_decay} to see that
\begin{multline*}\frac{U^{\frac{1}{2}}}{\tau^{\frac{1}{2}}\log\la\tau\ra} \|
  \partial Z^{\le 50}
  u_k\|_{L^2L^2(\tC^U_\tau)} \le
  \|r^{-\frac{1}{2}} \la t-r\ra^{-\frac{1}{2}}(\log\la r\ra)^{-1} Z^{\le 51} u_k\|_{L^2L^2(\tC^U_\tau)}
  \\+ \|r^{-\frac{1}{2}}\la t-r\ra^{-\frac{1}{2}}(\partial_t+\partial_r) (r Z^{\le 50} u_k)\|_{L^2L^2(\tC^U_\tau)}.
\end{multline*}
And thus,
\begin{equation}\label{IX}XI_k \lesssim I_k +VIII_k + VII_k.
\end{equation}
Then \eqref{I+II_goal} and \eqref{III_goal} give the boundedness in
terms of the right side of \eqref{MkGoal}.

\medskip
\noindent$\mathbf{[XII_k]}$:
Using \eqref{angBound}, we first note that
\begin{align*}\sum_\tau \frac{\tau}{(\log \la \tau\ra)^2} \|\good \partial Z^{\le
    40} &u_k\|_{L^2L^2(\tC^{U=1}_\tau)}^2
\\&\lesssim \sum_\tau \frac{1}{\tau (\log \la \tau \ra)^2} \Bigl(\|Z^{\le
  42} u_k\|_{L^2L^2(\tC^{U=1}_\tau)}^2+
\|(\partial_t+\partial_r)(rZ^{\le
  41}u_k)\|^2_{L^2L^2(\tC^{U=1}_\tau)}\Bigr)\\
&\lesssim VIII_k^2+VII_k^2.
\end{align*}

And for $1<U\le \tau/4$, \eqref{angBound} and \eqref{derivative_decay} give
\begin{align*}
  U^{\frac{1}{2}}\tau^{\frac{1}{2}} \|\good \partial Z^{\le 40} u_k\|_{L^2L^2(\tC^U_\tau)}&\lesssim
                                      \frac{U^{\frac{1}{2}}}{\tau^{\frac{1}{2}}}
                                      \|\partial Z^{\le 41}
                                      u_k\|_{L^2L^2(\tC^U_\tau)} +
                                      U^{\frac{1}{2}}\tau^{\frac{1}{2}}
                                      \|(\partial_t+\partial_r)(\partial_t-\partial_r)
                                      Z^{\le 40}
                                      u_k\|_{L^2L^2(\tC^U_\tau)}\\
&\lesssim \frac{U^{\frac{1}{2}}}{\tau^{\frac{1}{2}}} \|\partial Z^{\le
                                                                   41}
                                                                   u_k\|_{L^2L^2(\tC^U_\tau)}
                                                                   +
U^{\frac{1}{2}}\tau^{\frac{1}{2}}
                                                                   \|\Box
                                                                   Z^{\le
                                                                   40}
                                                                   u_k\|_{L^2L^2(\tC^U_\tau)}.
\end{align*}
Applying \eqref{cut} to the lower order terms of \eqref{lowOrderBox}
gives
\begin{multline*}
 U^{\frac{1}{2}}\tau^{\frac{1}{2}}\|\Box Z^{\le 40}
  u_k\|_{L^2L^2(\tC^U_\tau)}
\\\lesssim 
 \frac{1}{U^{\frac{1}{2}}\tau^{\frac{1}{2}}}\|Z^{\le
  41}u_{k-1}\|_{L^2L^2(\tC^U_\tau)}\Bigl(\frac{U^{\frac{1}{2}}}{\tau^{\frac{1}{2}}}\|\partial Z^{\le 24}
u_{k-1}\|_{L^2L^2(\ttC^U_\tau)} + U^{\frac{1}{2}}\tau^{\frac{1}{2}}\|\good  \partial Z^{\le
  23} u_{k-1}\|_{L^2L^2(\ttC^U_\tau)}\Bigr)
\\+  \frac{1}{U^{\frac{1}{2}}\tau^{\frac{1}{2}}}\|Z^{\le
  41}u_{k-1}\|_{L^2L^2(\tC^U_\tau)}\Bigl(\frac{U^{\frac{1}{2}}}{\tau^{\frac{1}{2}}}\|\partial Z^{\le 25}
u_{k}\|_{L^2L^2(\ttC^U_\tau)} + U^{\frac{1}{2}} \tau^{\frac{1}{2}}\|\good \partial Z^{\le
  24} u_{k}\|_{L^2L^2(\ttC^U_\tau)}\Bigr)
\\+\frac{1}{U^{\frac{1}{2}}\tau^{\frac{1}{2}}} \Bigl(\|Z^{\le 25}
  u_{k-1}\|_{L^2L^2(\ttC^U_\tau)} +
 \|(\partial_t+\partial_r)(r Z^{\le 24}
  u_{k-1})\|_{L^2L^2(\ttC^U_\tau)}\Bigr)\\\times \frac{U^{\frac{1}{2}}}{\tau^{\frac{1}{2}}}\Bigl(\|\partial Z^{\le
    40} u_{k-1}\|_{L^2L^2(\tC^U_\tau)} + \|\partial Z^{\le 41} u_k\|_{L^2L^2(\tC^U_\tau)}\Bigr).
\end{multline*}
From this, it follows that
\begin{multline}
  \label{X}
  XII_k \lesssim VIII_k + VII_k + XI_k + 
  VIII_{k-1}\Bigl(IX_{k-1}+X_{k-1}+   \log\la
  T\ra\Bigl(XI_{k-1}+XII_{k-1}\Bigr)\Bigr)
\\+ VIII_{k-1}\Bigl(IX_{k}+X_{k}+   \log\la T\ra\Bigl(XI_{k}+XII_{k}\Bigr)\Bigr)
+ \Bigl(\log\la T\ra VIII_{k-1} + VII_{k-1}\Bigr)\Bigl(XI_{k-1}+XI_k\Bigr).
\end{multline}
The proof of \eqref{MkGoal} is now complete as we can apply
\eqref{III_goal} and \eqref{IX}.

\medskip
\subsection{Convergence}
We now conclude the proof by showing that the sequence is Cauchy.
This is somewhat simplified by the extremely high regularity that was
used in the preceding subsection.  We set
\begin{equation}
  \label{Ak}
  A_k=\|\partial Z^{\le 20} (u_k-u_{k-1})\|_{L^\infty L^2} + \|Z^{\le 20} (u_k-u_{k-1})\|_{LE^1},
\end{equation}
and we shall show that for each $k$,
\begin{equation}
  \label{Ak_goal}
  A_k \le \frac{1}{2}A_{k-1},
\end{equation}
which suffices to complete the proof.

We use that
\begin{multline}
  \label{quadratic2}
  \Box u_k - \Box u_{k-1}
  = a^{I,\alpha}_{JK} (u_{k-1}^J-u_{k-2}^J)\partial_\alpha u_{k-1}^K
  + a^{I,\alpha}_{JK} u_{k-2}^J \partial_\alpha (u_{k-1}^K-u_{k-2}^K)
  \\+b^{I,\alpha\beta}_{JK} \partial_\alpha(u_{k-1}^J-u_{k-2}^J)
  \partial_\beta u_{k-1}^K + b^{I,\alpha\beta}_{JK} \partial_\alpha
  u_{k-2}^J \partial_\beta (u_{k-1}^K -u_{k-2}^K) \\+
  A^{I,\alpha\beta}_{JK}
  u_{k-1}^K\partial_\alpha\partial_\beta (u_k^J-u_{k-1}^J)
  +A^{I,\alpha\beta}_{JK}(u_{k-1}^K- u_{k-2}^K)
  \partial_\alpha\partial_\beta u_{k-1}^J
  \\
  +B^{I,\alpha\beta\gamma}_{JK} \partial_\gamma u_{k-1}^K
  \partial_\alpha\partial_\beta(u^J_k-u^J_{k-1})
  + B^{I,\alpha\beta\gamma}_{JK} \partial_\gamma (u_{k-1}^K-u_{k-2}^K)
  \partial_\alpha\partial_\beta u_{k-1}^J.
\end{multline}
With $h$ as in \eqref{h}, we apply \eqref{le_pert}.  Arguing as in
\eqref{product1}, though we need not take as much care to distinguish
the lower order terms from the higher order terms, this gives that
\begin{multline}\label{Aproduct}
  A_k^2 \lesssim \int_0^T\int |\partial^{\le 1}Z^{\le 20} (u_{k-1}-u_{k-2})| |\partial
  Z^{\le 21} u_{k-1}| \Bigl(|\partial Z^{\le 20}(u_k-u_{k-1})| +
  \frac{|Z^{\le 20} (u_k-u_{k-1})|}{r}\Bigr)\,dx\,dt
\\+ \int_0^T\int |\partial^{\le 1}Z^{\le 20} u_{k-2}| |\partial Z^{\le 20} (u_{k-1}-u_{k-2})|\Bigl(|\partial Z^{\le 20}(u_k-u_{k-1})| +
  \frac{|Z^{\le 20} (u_k-u_{k-1})|}{r}\Bigr)\,dx\,dt
\\+\int_0^T\int |\partial^{\le 1} Z^{\le 20} u_{k-1}| |\partial Z^{\le
  20} (u_k-u_{k-1})| \Bigl(|\partial Z^{\le 20}(u_k-u_{k-1})| +
  \frac{|Z^{\le 20} (u_k-u_{k-1})|}{r}\Bigr)\,dx\,dt.
\end{multline}

Applying \eqref{crt} and \eqref{ctt}, we obtain
\begin{multline*}
  \int\int_{C^R_\tau} |\partial^{\le 1}Z^{\le 20} (u_{k-1}-u_{k-2})| |\partial
  Z^{\le 21} u_{k-1}| \Bigl(|\partial Z^{\le 20}(u_k-u_{k-1})| +
  \frac{|Z^{\le 20} (u_k-u_{k-1})|}{r}\Bigr)\,dx\,dt\\\lesssim \Bigl(\|\partial Z^{\le
  25}u_{k-1}\|_{L^2L^2(\tC^R_\tau)} + R\|\good \partial Z^{\le
  24}u_{k-1}\|_{L^2L^2(\tC^R_\tau)}\Bigr)  \|\la
r\ra^{-\frac{3}{2}} \partial^{\le 1} Z^{\le 20}
(u_{k-1}-u_{k-2})\|_{L^2L^2(C^R_\tau)}\\\times
\Bigl(\|\la r\ra^{-\frac{1}{2}}\partial Z^{\le 20}
(u_k-u_{k-1})\|_{L^2L^2(C^R_\tau)} + \|\la r\ra^{-\frac{1}{2}}r^{-1}
Z^{\le 20}(u_k-u_{k-1})\|_{L^2L^2(C^R_\tau)}\Bigr).
\end{multline*}
And \eqref{cut} gives
\begin{multline*}
  \int\int_{C^U_\tau} |\partial^{\le 1}Z^{\le 20} (u_{k-1}-u_{k-2})| |\partial
  Z^{\le 21} u_{k-1}| \Bigl(|\partial Z^{\le 20}(u_k-u_{k-1})| +
  \frac{|Z^{\le 20} (u_k-u_{k-1})|}{r}\Bigr)\,dx\,dt
\\\lesssim
\frac{U^{\frac{1}{2}}}{\tau^{\frac{1}{2}}} \Bigl(\|\partial Z^{\le
  25}u_{k-1}\|_{L^2L^2(\tC^U_\tau)}+\tau\|\good \partial Z^{\le
  24}u_{k-1}\|_{L^2L^2(\tC^U_\tau)}\Bigr)
\|\la t-r\ra^{-1}\la r\ra^{-\frac{1}{2}} \partial^{\le 1} Z^{\le 20} (u_{k-1}-u_{k-2})\|_{L^2L^2(C^U_\tau)}
\\\times \Bigl(\|\la r\ra^{-\frac{1}{2}}\partial Z^{\le 20}
(u_k-u_{k-1})\|_{L^2L^2(C^U_\tau)} + \|\la r\ra^{-\frac{1}{2}}r^{-1}
Z^{\le 20}(u_k-u_{k-1})\|_{L^2L^2(C^U_\tau)}\Bigr).
\end{multline*}
We may use \eqref{lastHardy} and \eqref{L2L2toLE} to see that
\[\|\la t-r\ra^{-1}\la r\ra^{-\frac{1}{2}} \partial^{\le 1} Z^{\le 20}
  (u_{k-1}-u_{k-2})\|_{L^2L^2}
\lesssim (\log\la T\ra)^{\frac{1}{2}} \|Z^{\le
  20}(u_{k-1}-u_{k-2})\|_{LE^1}.\]
Thus, upon summing over $R\le \tau/2$, $U\le \tau/4$ and $\tau\le T$, we see that
\begin{multline}\label{A.1}
  \int_0^T\int |\partial^{\le 1}Z^{\le 20} (u_{k-1}-u_{k-2})| |\partial
  Z^{\le 21} u_{k-1}| \Bigl(|\partial Z^{\le 20}(u_k-u_{k-1})| +
  \frac{|Z^{\le 20} (u_k-u_{k-1})|}{r}\Bigr)\,dx\,dt
\\\lesssim (\log\la T\ra)^2 M_{k-1} A_{k-1} A_k.
\end{multline}

Applying \eqref{crt} (and \eqref{ctt}) and \eqref{cut}, respectively, give
\begin{multline*}
  \int\int_{C^R_\tau} |\partial^{\le 1}Z^{\le 20} u_{k-2}| |\partial Z^{\le 20} (u_{k-1}-u_{k-2})|\Bigl(|\partial Z^{\le 20}(u_k-u_{k-1})| +
  \frac{|Z^{\le 20} (u_k-u_{k-1})|}{r}\Bigr)\,dx\,dt
\\\lesssim \Bigl(\|r^{-1} Z^{\le 25} u_{k-2}\|_{L^2L^2(\tC^R_\tau)} +
\|\good Z^{\le 24} u_{k-2}\|_{L^2L^2(\tC^R_\tau)} \Bigr) \|\la
r\ra^{-\frac{1}{2}} \partial Z^{\le 20}
(u_{k-1}-u_{k-2})\|_{L^2L^2(C^R_\tau)} 
\\\times \Bigl(\|\la r\ra^{-\frac{1}{2}}\partial Z^{\le 20}
(u_k-u_{k-1})\|_{L^2L^2(C^R_\tau)} + \|\la r\ra^{-\frac{1}{2}}r^{-1}
Z^{\le 20}(u_k-u_{k-1})\|_{L^2L^2(C^R_\tau)}\Bigr)
\end{multline*}
and
\begin{multline*}
  \int\int_{C^U_\tau} |\partial^{\le 1}Z^{\le 20} u_{k-2}| |\partial Z^{\le 20} (u_{k-1}-u_{k-2})|\Bigl(|\partial Z^{\le 20}(u_k-u_{k-1})| +
  \frac{|Z^{\le 20} (u_k-u_{k-1})|}{r}\Bigr)\,dx\,dt
\\\lesssim \frac{1}{U^{\frac{1}{2}}\tau^{\frac{1}{2}}} \Bigl(\|Z^{\le 25}
u_{k-2}\|_{L^2L^2(\tC^U_\tau)} + \|(\partial_t+\partial_r)(rZ^{\le 24}u_{k-2})\|_{L^2L^2(\tC^U_\tau)}\Bigr)
\|\la r\ra^{-\frac{1}{2}} \partial Z^{\le 20} (u_{k-1}-u_{k-2})\|_{L^2L^2(C^U_\tau)}
\\\times \Bigl(\|\la r\ra^{-\frac{1}{2}}\partial Z^{\le 20}
(u_k-u_{k-1})\|_{L^2L^2(C^U_\tau)} + \|\la r\ra^{-\frac{1}{2}}r^{-1}
Z^{\le 20}(u_k-u_{k-1})\|_{L^2L^2(C^U_\tau)}\Bigr),
\end{multline*}
which, using \eqref{L2L2toLE}, imply 
\begin{multline}\label{A.2}
  \int_0^T\int |\partial^{\le 1}Z^{\le 20} u_{k-2}| |\partial Z^{\le 20} (u_{k-1}-u_{k-2})|\Bigl(|\partial Z^{\le 20}(u_k-u_{k-1})| +
  \frac{|Z^{\le 20} (u_k-u_{k-1})|}{r}\Bigr)\,dx\,dt
\\\lesssim (\log \la T\ra)^{2} M_{k-2} A_{k-1} A_k.
\end{multline}

The third term in \eqref{Aproduct} is very much of the same from as
the preceding term, and the exact same arguments yield

\begin{multline}\label{A.3}
  \int_0^T\int |\partial^{\le 1} Z^{\le 20} u_{k-1}| |\partial Z^{\le
  20} (u_k-u_{k-1})| \Bigl(|\partial Z^{\le 20}(u_k-u_{k-1})| +
  \frac{|Z^{\le 20} (u_k-u_{k-1})|}{r}\Bigr)\,dx\,dt\\
\lesssim (\log\la T\ra)^2 M_{k-1} A_k^2.
\end{multline}

It now follows from \eqref{Aproduct}, \eqref{A.1}, \eqref{A.2}, and
\eqref{A.3} that
\[A_k^2 \lesssim (\log\la T\ra)^2 (M_{k-1}+M_{k-2}) A_{k-1}A_k + (\log \la
  T\ra)^2 M_{k-1} A_k^2.\]
Using \eqref{MkGoal} and \eqref{lifespan}, provided $c\ll 1$, we may
bootstrap and obtain
\[A_k^2\lesssim c^4 \varepsilon^{\frac{2}{3}} A^2_{k-1}.\]
Thus, for $\varepsilon$ sufficiently small, we recover \eqref{Ak_goal},
which implies that the sequence is Cauchy and thus convergent.  This
completes the proof.


\bibliography{exterior}

\end{document}